\documentclass{amsart}[12pt]
\usepackage{amsmath, amssymb}
\usepackage{amsfonts, mathrsfs}
\usepackage{amsthm, upref}
\usepackage{graphicx}
\usepackage[usenames, dvipsnames]{color}

\newcommand{\comment}[1]{}
\newcommand{\eq}{\begin{equation}}
\newcommand{\en}{\end{equation}}
\newcommand{\rr}{\mathbb{R}}

\newcommand{\abs}[1]{\left\lvert #1 \right\rvert}

\newcommand{\combi}[2]{\begin{pmatrix}#1 \\ #2 \end{pmatrix}}

\newcommand{\Root}{\rho}
\newcommand{\ances}{\preccurlyeq}
\newcommand{\rtspace}{\mathfrak{RTree}}
\newcommand{\wrtspace}{\mathfrak{RTree}^*}
\newcommand{\dis}{\text{dis}}
\newcommand{\mx}{\mathcal{N}}
\newcommand{\borel}{\mathcal{B}}

\newcommand{\regtran}{\mathfrak{H}}

\newcommand{\discretetree}{\mathcal{T}}
\newcommand{\tree}{\mathbb{T}}
\newcommand{\rtree}{\mathfrak{T}}
\newcommand{\chtree}{\mathbb{U}}
\newcommand{\treesp}{\mathbf{\tau}}
\newcommand{\dtree}{\mathbf{d}}

\newcommand{\splitlaw}{\mathcal{P}}
\newcommand{\leb}{\mathbb{LEB}}

\newcommand{\lbar}{\overline{L}}
\newcommand{\rrp}{\mathbb{R}_+}
\newcommand{\meas}{\mathcal{M}_+}

\newcommand{\discretetreesp}{\mathcal{U}}
\newcommand{\proj}{\mathbb{P}}

\newcommand{\gbar}{\overline{G}}
\renewcommand{\hbar}{\overline{H}}
\newcommand{\gw}{\text{GW}}
\newcommand{\age}{\text{Age}}
\newcommand{\esssup}{\text{ess sup}}
\newcommand{\phibar}{\overline\Phi}

\newcommand{\dual}{\widehat}
\newcommand{\street}{\text{Street}}
\newcommand{\streetmarginal}{\Gamma}
\newcommand{\hugelaw}{\Theta}

\begin{document}

\theoremstyle{plain}
\newtheorem{thm}{Theorem}
\newtheorem{lemma}[thm]{Lemma}
\newtheorem{prop}[thm]{Proposition}
\newtheorem{cor}[thm]{Corollary}

\theoremstyle{definition}
\newtheorem{defn}{Definition}
\newtheorem{asmp}{Assumption}
\newtheorem{notn}{Notation}
\newtheorem{prb}{Problem}

\theoremstyle{remark}
\newtheorem{rmk}{Remark}
\newtheorem{exm}{Example}
\newtheorem{clm}{Claim}

\title[]{On the Aldous diffusion on Continuum Trees. I}

\author{Soumik Pal}
\address{Department of Mathematics\\ University of Washington\\ Seattle, WA 98195}
\email{soumik@u.washington.edu}


%\subjclass[2000]{}

\thanks{This research is partially supported by NSF grant DMS-1007563}

\date{\today}

\begin{abstract} Consider a Markov chain on the space of rooted real binary trees that randomly removes leaves and reinserts them on a random edge and suitably rescales the lengths of edges. This chain was introduced by David Aldous who conjectured a diffusion limit of this chain, as the size of the tree grows, on the space of continuum trees. We prove the existence of a process on continuum trees, which via a random time change, displays properties one would expect from the conjectured Aldous diffusion. The existence of our process is proved by considering an explicit scaled limit of a Poissonized version of the Aldous Markov chain running on finite trees. The analysis involves taking limit of a sequence of splitting trees whose age processes converge but the contour process does not. Several formulas about the limiting process are derived. 
\end{abstract}

\maketitle

\tableofcontents

\section{Introduction}\label{sec:intro}

A Cladogram is an unrooted semi-labeled binary tree. That is to say, every vertex that is not of degree one must have degree three. The vertices of degree one are called leaves which are labeled by $\{ 1,2,\ldots, n \}$, while every other vertex will be called \textit{internal vertices}. It follows that if such a tree has exactly $n$ leaves, it must have $2n-3$ edges. In \cite{AOP,A00} Aldous considers the following Markov chain on the space of $n$-leaf Cladograms. For a precise description we first define two operations on Cladograms closely following \cite{A00}, to which we refer the readers for more details.
\begin{enumerate}
\item[(i)] To \textit{remove a leaf} $i$. The leaf $i$ is attached by an edge $e_1$ to a branchpoint $b$ where two other edges $e_2$ and $e_3$ are incident. See Figure \ref{fig_clado}. Delete edge $e_1$ and branchpoint $b$, and then merging the two remaining edges $e_2$ and $e_3$ into a single edge $e$. The resulting tree has $2n-5$ edges. 
\item[(ii)] To \textit{add a leaf} to an edge $f$. Create a branchpoint $b'$ which splits the edge $f$ into two edges $f_2, f_3$ and attach the leaf $i$ to branchpoint $b'$ via a new edge $f_1$. This restores the number of leaves and edges to the tree. 
\end{enumerate}

Let $\treesp_n$ denote the finite collection of all $n$-leaf Cladograms. Let $\tree, \tree'$ be elements on $\treesp_n$. Write $\tree' \sim \tree$ if $\tree'\neq \tree$ and $\tree'$ can be obtained from $\tree$ by following the two operations above for some choice of $i$ and $f$. Thus a $\treesp_n$ valued chain can be described by saying: remove leaf $i$ uniformly at random, then pick edge $f$ at random and re-attach $i$ to $f$. In particular its transition matrix is
\[
P(\tree, \tree')=\begin{cases}
\frac{1}{n(2n-5)},& \quad \text{if}\quad \tree'\sim \tree\\
\frac{n}{n(2n-5)},& \quad \text{if}\quad \tree'=\tree.
\end{cases}
\]
This leads to a symmetric, aperiodic, and irreducible finite state space Markov chain. Schweinsberg \cite{S} proved that the relaxation time for this chain is $O(n^2)$, improving a previous result in \cite{A00}. We will refer to this Markov chain as the Aldous Markov chain in the following text.

The problem of interest for us in this article is to establish a limiting diffusion for this sequence of Markov chains as $n$ tends to infinity. It is not entirely rigorous at this point what we mean by a limit here. For example, what is our limiting state space? Detailed answers of such questions will be provided later. Roughly, our state space will be the space of \textit{continuum trees}, i.e., compact metric spaces without loops which support a probability measure on them.  Hence we are looking for a \textit{continuous} Markov process on the space of continuum trees that can be argued as a weak limit of the above sequence of chains.

There are some technical problems of working with an unrooted Cladogram. Hence, we will make a small change in the above model and consider Cladograms in which a distinguished vertex in the tree at time zero is marked as root which is never removed and continues as the root of all subsequent trees. Instead of calling Cladograms, we will simply call such trees as rooted binary trees. Please see Section \ref{sec:prelimtrees} to find formal definitions of the kind of trees we work with.  

When trees are rooted, it is natural to associate a family structure with it where the root is an ancestor, and every vertex gives rise to a number of children / progeny. We will frequently use this analogy whose meaning will be obvious from the context.

\begin{figure}[t]
\centering
\includegraphics[width=5in, height=2.5in]{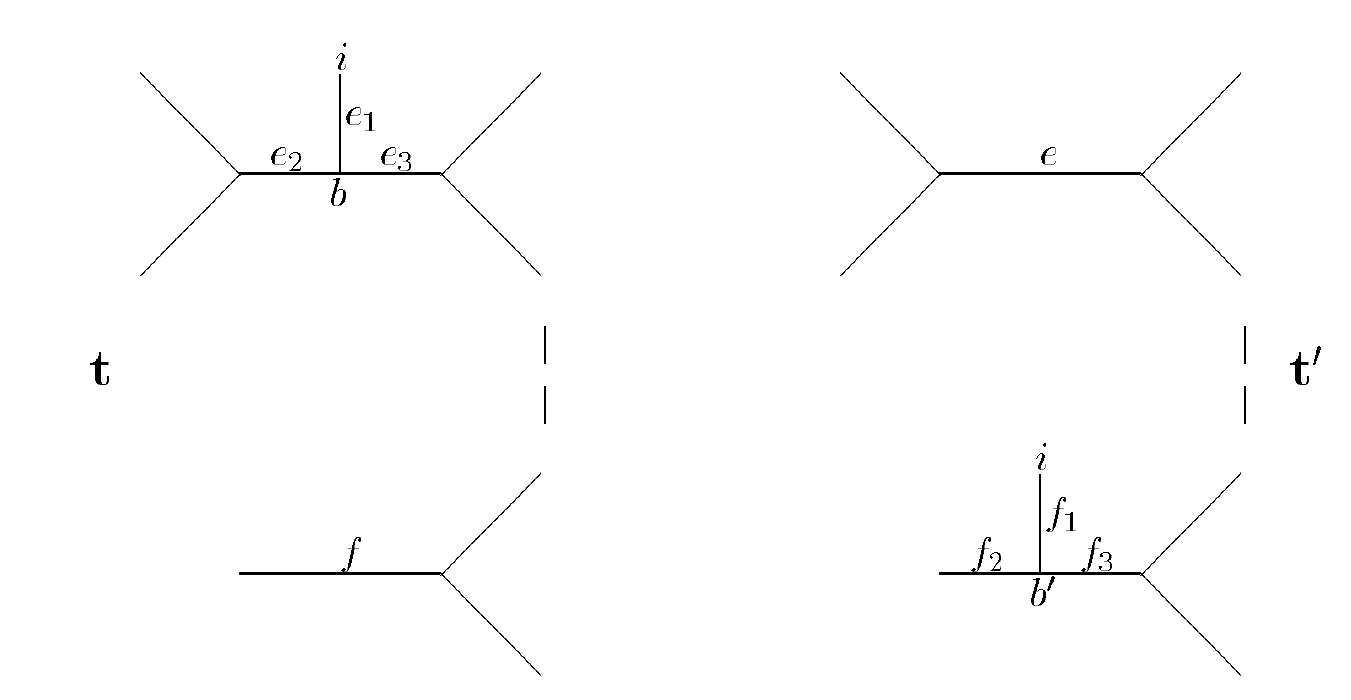}
\caption{Markov chain on Cladogram (reproduced from \cite{A00})}
\label{fig_clado}
\end{figure}

\subsection{Poissonization}

We now introduce the important concept of \textit{Poissonization} of the Markov chain described above. One challenging aspect of the Aldous Markov chain on Cladograms is that at every step of the chain the number of leaves remain the same. Things  become simpler if we adopt the Poissonized version of the Markov chain: every leaf has an exponential clock with rate $2$ attached to it which determines the time of  their removal (deaths). Every edge also has an exponential clock of rate $1$ at which point a new pair of vertices (an internal point and a leaf connected by a edge) is introduced on that edge. The rate of births and deaths are chosen to accommodate for the fact that there are about twice as many edges as leaves. We will refer to this model as the Poissonized Aldous Markov chain. 

This Poissonization changes the Markov chain when the number of leaves are finite. However, there are strong reasons to believe that the limiting diffusion for the non-Poissonized chain is equal in law the limiting diffusion of the Poissonized chain via a stochastic change of time.  We provide an argument below to make our case. 

Going back to the Aldous Markov chain (unrooted), consider a branch point $b$ in the tree $\tree$. It divides the collection of leaves naturally into three sets. Let $X(\tree)=(X_1, X_2, X_3)(\tree)$ denote the vector of proportion of leaves in each set. It can be verified (see Pal \cite{Pal2}) that the Markov chain on Cladograms induce a Markov chain on these vector of proportions, which, when run at $n^2/2$ speed, converges in law to a diffusion on the state space $\{ (x_1, x_2, x_3),\; x_i\ge 0,\; x_1+x_2+x_3=1  \}$. The law of this diffusion is called the Negative-Wright-Fisher or NWF$(1/2,1/2,1/2)$.

In Pal \cite{Pal2} it is shown that the NWF law can be obtained by considering ratios of squared Bessel processes of negative dimensions. A comprehensive treatment of BESQ processes can be found in the book by Revuz \& Yor \cite{RY}. This family of one dimensional diffusions is indexed by a single real parameter $\theta$ (called the dimension) and are solutions of the stochastic differential equations
\eq\label{besqintro}
Z(t)= x + 2 \int_0^t \sqrt{{Z(s)}}d\beta(s) + \theta t, \qquad x \ge 0, \quad t\ge 0,
\en
where $\beta$ is a one dimensional standard Brownian motion. We denote the law of this process by $Q^\theta_x$. It can be shown that the above SDE admits a unique strong solution until it hits the origin. The classical model only admits paramater $\theta$ to be non-negative. However, an extension, introduced by G\"oing-Jaeschke \& Yor \cite{yornbesq}, allows the parameter $\theta$ to be negative. 

It is well known that BESQ process with a nonnegative dimension $\theta$ is the diffusion limit, as the number of initial particles go to infinity, of a critical (discrete or continuous time) Galton-Watson tree with a rate of immigration $\theta$. For example, this happens for a continuous time, binary Galton-Watson, every individual dies or reproduces at rate $2$ and there is an influx of new members at rate $\theta$. It follows from the same proof that when $\theta$ is negative, immigration changes to emigration, i.e., at rate $\theta$, members leave the family until there is no one surviving. 

The following construction of NWF processes can be found in \cite[Theorem 8]{Pal2}.

\begin{lemma}\label{bestimechange}
Let $Z=(Z_1, Z_2, Z_3)$ be a vector of $3$ iid BESQ processes of dimension $-1$. Let $\zeta$ be the sum $\sum_{i=1}^3 Z_i$. 
Define 
\[
T_i= \inf\left\{ t\ge 0: \; Z_i(t)=0  \right\}, \quad \tau=\wedge_{i=1}^3 T_i.
\]
Then there is a $NWF(1/2,1/2,1/2)$ process $(\mu_1, \mu_2, \mu_3)$ for which the following equality holds:
\eq\label{bestime}
\mu_i\left( 4C_t \right)= \frac{Z_i}{\zeta}(t\wedge \tau), \quad 1\le i \le 3,\qquad C_t=\int_0^{t\wedge \tau} \frac{ds}{\zeta(s)}.
\en
\end{lemma}

A very similar result holds when instead of considering one branch point we consider several branch points and look into the partition of leaf masses generated by these branch points. The vector of proportions admits a diffusion limit which can obtained via a random time change from independent BESQ processes (of possibly both positive and negative dimensions). Hence it is natural to search for a Markov chain on Cladograms such that leaf masses in partitions generated by branch points converge to BESQ processes. This is exactly what happens for the Poissonized Markov chain since the number of leaves in each partition evolves as an independent Galton-Watson branching process with immigration or emigration. 

Hence, Theorem \ref{besqintro} gives us the important clue as to the idea of poissonization of the Aldous' chain and the correct value of the rate of emigration which is crucial in the following. More examples of Poissonization can be found in Markov chains like the Wright-Fisher chain whose Poissonized version is a branching process with immigration. The diffusion limit of the former is the Wright-Fisher diffusion and the diffusion limit of the latter is the BESQ process. The time-change relationship between the two can be found in Pal \cite{Pal1}. More reminiscent is the Fleming-Viot model where the Superprocess is the Poissonized version, and the theorem that allows us to obtain the Fleming-Viot from the Superprocess by a time and mass change is called the Perkins Disintegration Theorem (see Etheridge \cite{etheridge}). 

We end this discussion with an obvious lemma.

\begin{lemma}
Consider an internal vertex on a finite tree on which the Poissonized Aldous chain is running. This internal vertex divides the leaves of the tree in three subtrees rooted at that vertex. One of them contains the original root. Consider the number of leaves in any of the other two parts. The process of total number of leaves in these subtrees follow independent continuous time binary Galton-Watson branching process with a birth and death rate of two and an emigration rate $1$ until one of them hits zero.
\end{lemma}

\subsection{Coding trees by age} A fundamental concept at the heart of our argument is the so-called age of a random tree. Consider any model of generating a random tree starting with one individual. If after $a$ units of time later the tree survives, then the age of the tree is $a$. This is particularly useful when we have one time-axis and several trees start sprouting at different points in time. At any given instance of time, we consider the set of trees surviving at that time. Their ages are given by the durations since they have sprouted.   

In this dynamic situation properties of trees can be described by its age. For example, suppose that the total existing population in a dynamically growing random tree is given by a continuous times Galton-Watson branching process. Then, given its age, one can deduce its size from the conditional distributions of Galton-Watson branching process conditioned to survive up to time $a$.

Now, one can actually code an entire tree by the ages of its various subtrees. We will discuss these ideas in much more detail later in the text. However, to give an idea of our strategy, we provide a simplified informal description below.

\begin{figure}[t]
\centering
\includegraphics[width=4.5in, height=2.5in]{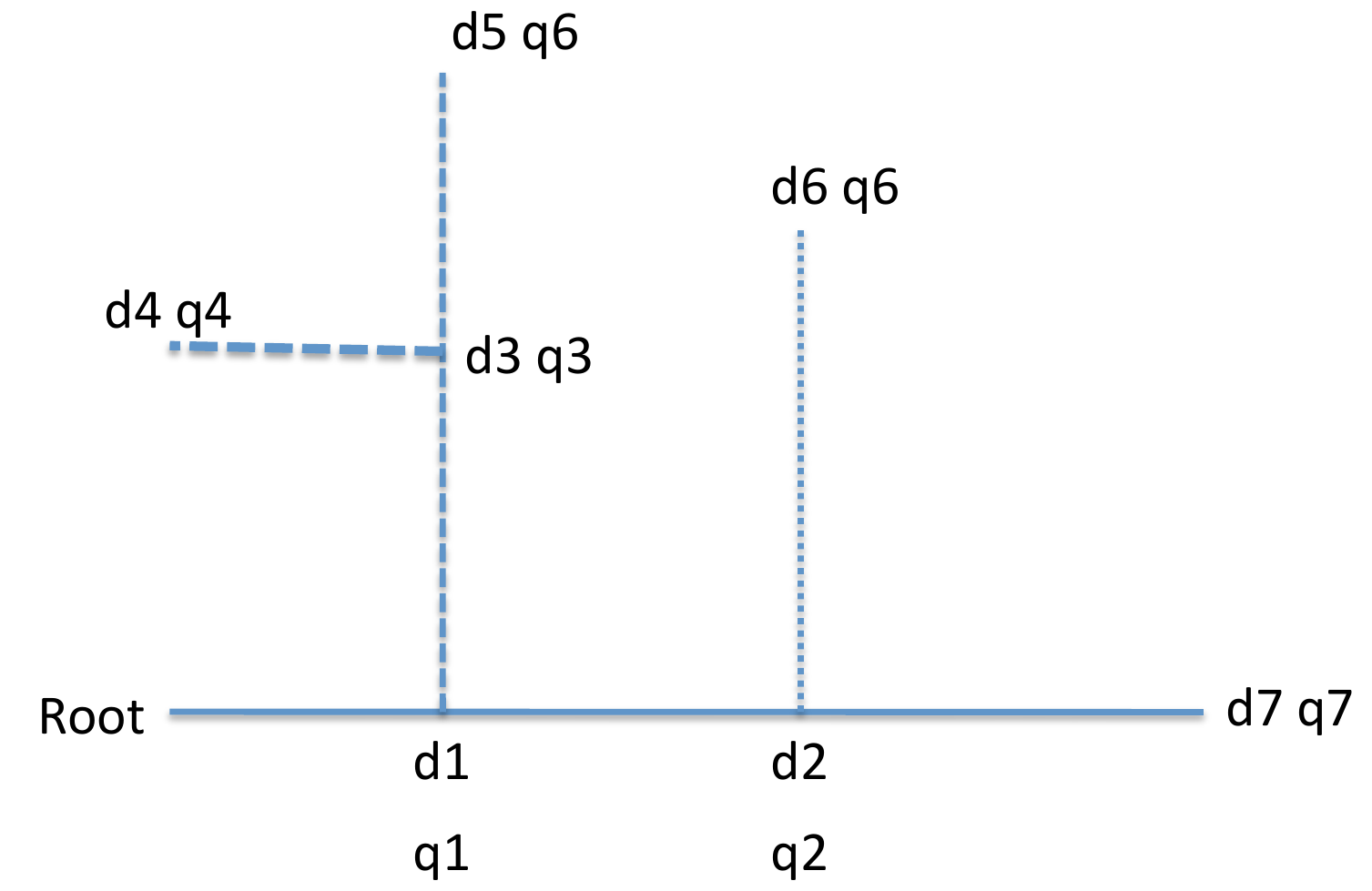}
\caption{Coding of a tree}
\label{introcode}
\end{figure}
\bigskip
\bigskip

Consider the image of a finite tree in Figure \ref{introcode}. The variables $\{ d_1, \ldots, d_7\}$ will be considered (positive) distances and the variables $\{ q_1, \ldots, q_7  \}$ will be considered as (nonnegative) masses. Now we construct a stochastic process $(A_1, A_2, A_3)$ as follows: $A_1$ takes three values $d_1 < d_2 < d_7$ with probabilities proportional to $q_1, q_2$, and $q_7$ respectively. Given $A_1$ the value of $A_2$ is determined. 
\begin{enumerate}
\item[Case (i)] Suppose $A_1=d_1$, we imagine our process \textit{enters} the first subtree which is marked with dashes. Then, $X_2$ takes two values $d_3 < d_5$ with probability proportional to $q_3$ and $q_6$ respectively. 
\item[Case (ii)] Suppose $A_1=d_2$, then our process enters the dotted subtree, and $A_2$ is $d_6$ with probability one, irrespective of $q_6$. 
\item[Case (iii)] Suppose $A_1=d_7$, then $A_2=0$ with probability one. 
\end{enumerate}

The reader can now guess the distribution of $A_3$, given the values of $A_1, A_2$. In fact $A_3=d_4$ with probability one if $\{ A_1=d_1, A_2=d_3\}$, or zero otherwise. 

The importance of this stochastic process is that suppose the tree is absent to start with, and we are given the distances and masses for the stochastic process only. It is easy in a natural way (to be formalized later) to construct the tree out of the set of possible paths of this stochastic process.  

Somewhat more formally, one can construct a set representation of the tree (according to the definition in Aldous \cite{A93}) from the support of the paths of the stochastic process. Note that, the measure representation of that set will not be given by the law of the stochastic process. By that we mean the distinct paths of the stochastic process are not equiprobable. Although it is not crucial in what we do, some discussion of this phenomenon needs to be done. Suppose we pick a leaf at random from the tree given in Figure \ref{introcode} and consider the sequence of $d_i$'s that appear ordered away from the root. Then, this provides another law of $(A_1, A_2, A_3)$ which pick subtrees according to their leaf mass. However, one can easily choose $q_i$'s such that a subtree with a small $q_i$ has a large number of leaves. 

In this article, as the tree evolves, new subtrees will be born and old subtrees will die. The masses $q_i$'s will be given by the ages of the subtrees. Note that if the sizes of the subtrees grow according to a Galton-Watson law, it is quite possible that an older subtree has much less leaves than a younger one.

Nevertheless we have made our case that a stochastic process whose coordinates are discrete random variables can code a tree in such a way that probability masses of entering a subtree are given by the age of the subtree. The leaf masses of the subtree can be recovered from their ages, and hence the two measures (given by age and leaf-mass) will be absolutely continuous with respect to one another. In particular, they will have the same support representing the tree itself.

\subsection{The content in this paper}

The title of this article reflects the fact that several loose ends in this article will be settled in a follow-up article. A precise summary of what has been achieved in this paper and what will be done in the follow-up is presented below. 
\bigskip

\noindent\textit{What has been done in this article:}

\begin{enumerate}
\item[(i)] We study the Poissonized version of the Aldous chain on Cladograms and consider its natural limit on the space of real trees. 

\item[(ii)] We code trees by a stochastic process whose law is random. This is a different way of looking at the $\ell^1$ set representations due to Aldous.  

\item[(iii)] Consider the (random) law of the first coordinate of the said stochastic process. As the Poissonized chain runs, it induces a Markov chain on point processes on $\rr^+ \times \rr^+$ which ultimately hits the empty point process with probability one. We establish finite-dimensional convergence of an excursion of this point process to a limiting process of point processes as time (of the Markov chain), distance (on the tree), and mass (leaf-masses on the trees) are properly scaled. This limiting point of point process is very explicit. In fact, the marginal distribution and the transition mechanism are all obtained as explicit formulas.

\item[(iv)] By using recursion one can similarly argue the (finite-dimensional) convergence of an excursion of the induced Markov chain on the joint law of finitely many coordinates of the stochastic process to a limiting law. 

\item[(v)] Now, we turn to constructing random trees coded according to these limiting laws at various times. We consider two time points and produce a joint coupling of random trees at these two time points whose corresponding stochastic processes are obtained via the limiting transition mechanism. An easy induction produces coupling at finitely many time points. As before, the formulas are all explicit. 
\end{enumerate}
\bigskip

\noindent\textit{What we plan to do in a follow-up article:}

\begin{enumerate}
\item[(i)] Prove that our couplings produce a Markov process on real trees and that one can choose it to have continuous paths in the Gromov-Hausdorff distance. 
\item[(ii)] That the Brownian CRT is an invariant distribution for the above Markov process.  
\end{enumerate}

\bigskip

\begin{figure}[t]
\centering
\includegraphics[width=4in, height=2in]{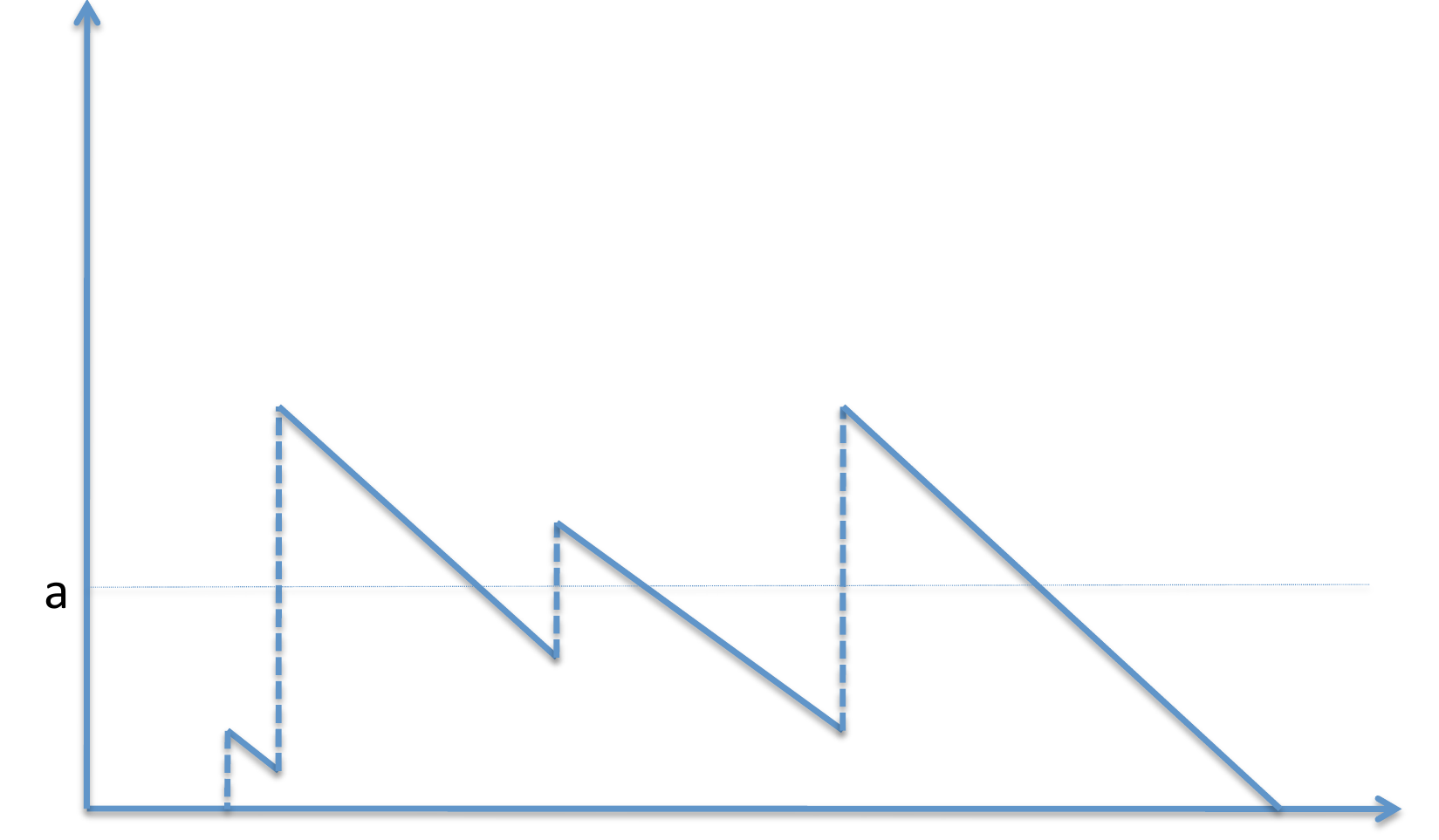}
\caption{Excursion of a reflected L\'evy process}
\label{backlevy}
\end{figure}
\bigskip

\subsection{Outline of the strategy}

Our strategy is the following. Consider again the tree in Figure \ref{introcode}. Assume that each edge is of unit length.
One can think of the horizontal line from the root as a spine on which two subtrees are supported. As the Poissonized Aldous Markov chain runs, new subtrees appear on the spine and old subtrees disappear.  The distances of these subtrees from the root is given by the index in which they appear away from the root. Thus at any instance of time one obtains a point process in $\rr^+ \times \rr^+$ recording the the distance $d$ and the age $q$ of the subtrees ordered away from the root. Hence the Poissonized chain induces a Markov process on point processes on $\rr^+ \times \rr^+$. If we find the limiting transition probabilities of this Markov process, this gives us how the law of $A_1$ changes with time. The process of laws of $A_2, A_3$, and so on are given by the same transition mechanism due to the obvious independence of leaf growth and deaths in various subparts of the tree.

Hence our central problem reduces to identifying the Markov process of point process of ages as determined by the Poissonized Aldous chain. This process has the following description in terms of a compound Poisson process with drift. Let $N$ be a Poisson process with rate one and let $\zeta_1, \zeta_2, \ldots$ be an iid sample taken from the distribution function $L(u) = 1 - (1 + 2u)^{-3/2}$, $u \ge 0$.

\begin{figure}[t]
\centering
\includegraphics[width=4.5in, height=3in]{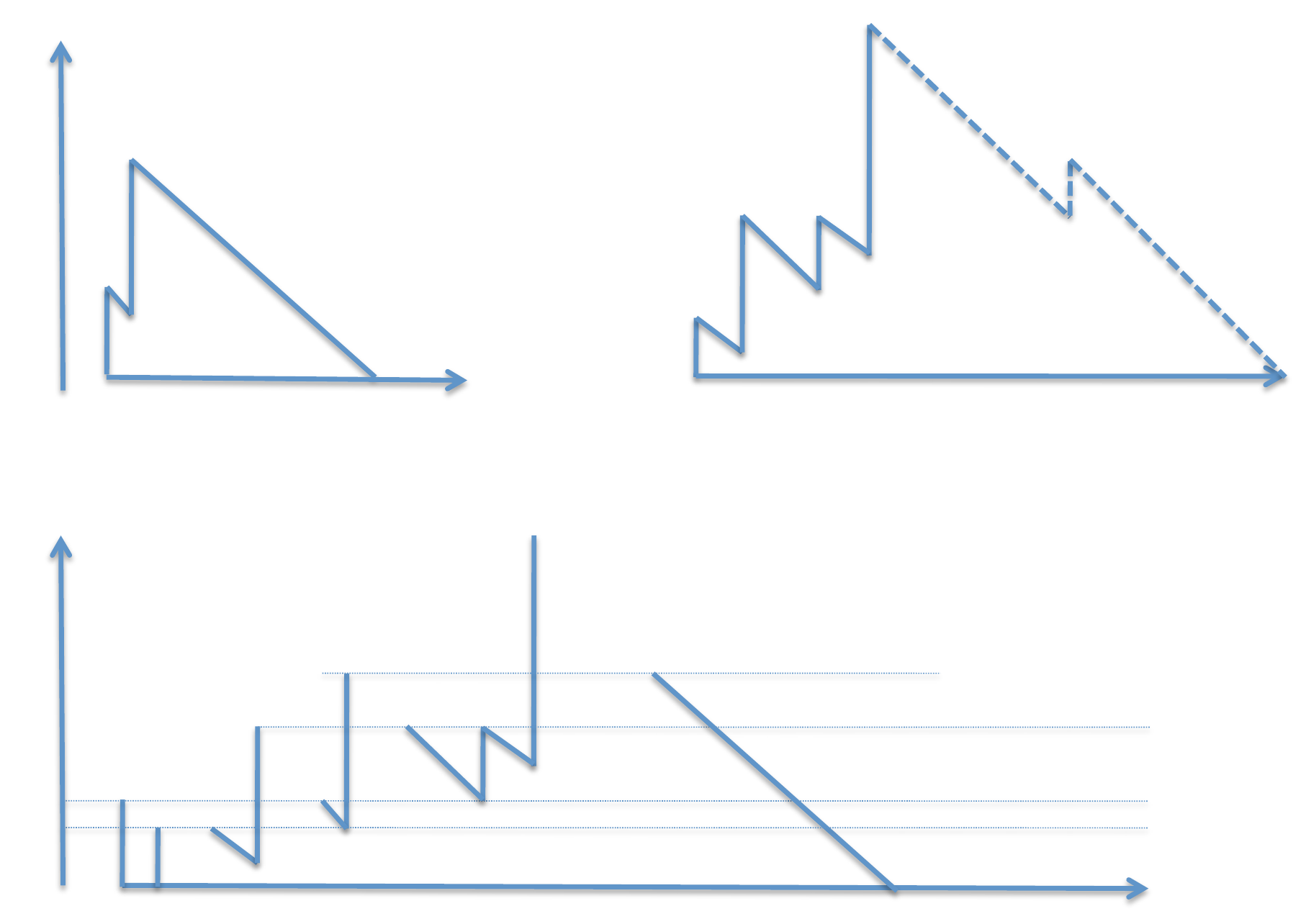}
\caption{Excursion of the age process}
\label{intropath}
\end{figure}

Define the compound Poisson process with drift
\[
X_t = -t + \sum_{i=1}^{N_t} \zeta_i, \qquad t\ge 0,
\]
starting at zero until it hits zero. Let $I$ denote the process of running infimum. That is
\[
I_t = \inf \left\{ X_s, \; 0\le s\le t \right\}, \quad t\ge 0.
\]
Consider the process $X-I$ which is $X$ reflected at its infimum. Then this reflected process can be seen as iid sequence of excursions above the origin. 

A typical image of such an excursion is shown in Figure \ref{backlevy}. For every {nonnegative} level $a$ on the $y$-axis, consider the sequence of jumps across level $a$. These are marked by dotted lines. Each of these jumps have an undershoot below level $a$, i.e., difference between $a$ and where the process was right before it jumps across $a$. The sizes of these jumps ordered by the indices in which they appear form a point process in $\rr^+ \times \rr^+$. A moments reflection will show the reader that as we increase $a$, starting from $0$, one obtains a Markov process on point process in $\rr^+ \times \rr^+$. 

Now, under proper scaling, the jump distribution of this L\'evy process does not converge to any nontrivial rate function. However, the renewal measure does, and hence the point process of undershoots indexed by their order of appearance has a certain limit. Notice that the distance on the $x$-axis plays no role in this analysis.

\bigskip

Our limiting point process on ages of subtrees is very closely related to this limit in the following way. 
Consider Figure \ref{intropath}. We have here two independent instances $X_0$ and $X_1$ of excursions of the above reflected L\'evy process. Now, we interweave the paths of $X_0$ and $X_1$ in the following manner. Start with the first jump of $X_0$, then concatenate the path of $X_1$ until it exceeds the jump of $X_0$. Then $X_1$ has reached a new height. Concatenate the remaining path of $X_0$ from where we left until it exceeds the height of $X_1$. Alternate between the two processes until one of them hits zero and we stop. 

As before, at every nonnegative level, one can consider the point process given by undershoots below that level. Except now we reverse the order in which these undershoots appear. That is to say, the farthest undershoot from the $y$-axis will be labeled one, and the closest will have the maximum label. This family of point processes, indexed by level, provides the transition probabilities of the ages of subtrees on a spine in the Poissonized Aldous Markov chain on finite trees. When we take a scaling limit, we obtain an excursion measure on paths from $(0,\infty)$ to the space of point processes on $\rr^+ \times \rr^+$. This excursion measure completely determines the dynamics of the limiting age process and hence the tree itself by the previous argument. 
\bigskip

In a nutshell our construction of the limiting process is the following: given any two time points $a_0 < a_1$, we construct a joint set-representation of a pair of continuum trees at those times whose transition mechanism is shown to be a scaling limit of the transition mechanism of the finite Poissonized Aldous chain. This set representation is achieved by a joint coupling of an iid sequence of random sequences at times $a_0$ and $a_1$ that produces consistent and leaf-tight families of proper $k$-trees, and hence should determine the limiting continuum trees by taking a closure in $\ell^1$. Unfortunately, at this point we cannot prove that the sequences remain in $\ell^1$ although it is intuitively obvious from the calculations. 

The reader should note that there are three kinds of scaling parameters involved. One, is the natural distance on the finite tree which converges to the Lebesgue distance on the continuum limit. Two, is the density of leaves, that also converges to a probability measure on the continuum trees. Three, is the factor by which the Markov chain has to be speeded up to reach a diffusion limit. The relationship between the three is the following: when the speed is of order $n$, the distance should be of the order $\sqrt{n}$, and the leaf-mass should be scaled by $n$. The fact that this is the only possible scaling can be verified by our discussion of NWF and BESQ processes. However, since we are not rescaling total leaf-mass to one, our continuum trees have leaf-mass measures that are finite positive measures. In fact, under the Poissonized diffusion, the total mass process evolves as a BESQ process of dimension $-1$. In particular, the process of trees ultimately dies. Note that, this total mass process is required in the rescaling the triplet (time, distance, mass) to obtain the Aldous diffusion from the limiting Poissonized diffusion. The stochastic change of time is given in \eqref{bestime}, where $\zeta$ is the total mass process. Furthermore, when we rescale the leaf-mass measure to one, one needs to rescale the Lebesgue distance on the tree by the square-root of the same factor.

\bigskip

The of scaling-limit approach used in this paper is very different more the Dirichlet form technology used in the study of jump Markov processes on continuum trees. For example, the process inspired by the subtree prune and regraft (SPR) as studied in the seminal article by Evans and Winter \cite{EW}. In an SPR move, a binary tree $T$ is cut ``in the middle of an edge''  to give rise to two subtrees, say $T'$ and $T''$. Another edge is chosen in $T'$, a new vertex is created ``in the middle'' of that edge and the cut edge in $T''$ is attached to this new vertex. Finally, the ``pendant'' cut edge in $T'$ is removed along with the vertex it was attached to in order to produce a new binary tree that has the same number of vertices as $T$. Please see \cite{EW} for a figure. 

This produces a \textit{macro-level} change in the continuum tree and hence is reflected by a pure-jump Markov process on the space of real trees.

The authors in \cite{EW} use Dirichlet form technique to prove the existence of a process on the space of continuum trees that mirrors the behavior of this discrete chain. One problem with that approach is that there is no convergence theorem from the Markov chain to the limiting process.

\section{Preliminaries on trees}\label{sec:prelimtrees}

\subsection{The space of discrete trees} We start with some notations. Let $\mathbb{N}$ denote the set of positive integers. Locally finite rooted trees are usually labeled and ordered by a finite sequence of integers as follows. The root of the tree is $0$. The $j$th child of $u=(u_1, u_2, \ldots, u_n)\in \mathbb{N}^n$, is $uj$, where for two elements $v \in \mathbb{N}^k$ and $w \in \mathbb{N}^j$, the element $vw$ belongs to $\mathbb{N}^{k+j}$ and is the concatenation of the component sequences. Let $\abs{u}=n$ be the generation, or genealogical height, of $u$. Let
\eq\label{whatiscalu}
\discretetreesp := \bigcup_{n=0}^\infty \mathbb{N}^n, \qquad \text{where}\quad \mathbb{N}^0=\{0\}. 
\en

Formally a discrete tree is a subset $\discretetree \subseteq \discretetreesp$ such that:
\begin{enumerate}
\item[(i)] $0\in \discretetree$.
\item[(ii)] If $v=uj\in \discretetree$ and $j \in \mathbb{N}$, then $u \in \discretetree$.
\item[(iii)] For every $u\in \discretetree$, there is a nonnegative integer $K_u \le \infty$ (the number of offsprings) such that $uj \in \discretetree$ if and only if $j \in \{1,2,\ldots, K_u\}$.
\end{enumerate}

We write $u \prec v$ if $u$ is an ancestor of $v$, i.e., there is a sequence $w$ such that $v=uw$. We denote by $u\wedge v$ the most recent common ancestor of $u$ and $v$.

\subsection{Real trees} The formalism for real trees is somewhat intricate and lengthy. We do not deal with all the details in this article and refer the reader to the excellent survey by Le Gall \cite{LGsurv} from which we borrow the following description. Also see the articles by Evans and Winter \cite{EW}, Evans-Pitman-Winter \cite{EPW}, and the preprint by Depperchmidt et al.~ \cite{DGP} where extensive formalism for dealing with Markov chains on real trees have been defined. We only consider compact real trees which can be defined as follows. 

\begin{defn} A compact metric space $(\rtree, \dtree)$ is a real tree is the following two properties hold for every $a,b \in \rtree$.
\begin{enumerate}
\item[(i)] There is a unique isometric map $f_{a,b}$ from $[0, \dtree(a,b)]$ into $\rtree$ such that $f_{a,b}(0)=a$ and $f_{a,b}(\dtree(a,b))=b$.
\item[(ii)] If $q$ is a continuous injective map from $[0,1]$ into $\rtree$, such that $q(0)=a$ and $q(1)=b$, we have
\[
q\left( [0,1] \right) = f_{a,b} \left( [0, \dtree(a,b)] \right).
\]
\end{enumerate}
A rooted real tree is a real tree $(\rtree, \dtree)$ with a distinguished vertex $\Root=\Root(\rtree)$ called the root. 
\end{defn}

Let us consider a rooted real tree $(\rtree , \dtree)$. The range of the mapping $f_{a,b}$ in (i) is denoted by $[[a, b]]$ (this is the line segment between $a$ and $b$ in the tree). In particular, $[[\Root, a]]$ is the path going from the root to $a$, which we will interpret as the ancestral line of vertex $a$. More precisely we define a partial order on the tree by setting $a\ances  b$ ($a$ is an ancestor of $b$) if and only if $a \in [[\rho,b]]$. If $a,b \in \rtree$, there is a unique $c\in \rtree$ such that $[[\rho,a]]\cap[[\rho,b]] = [[\rho,c]]$. We write $c = a \wedge b$ and call $c$ the most recent common ancestor to $a$ and $b$. By definition, the multiplicity of a vertex $a \in \rtree$ is the number of connected components of $\rtree \backslash \{a\}$. Vertices of $\rtree \backslash \{\rho\}$ which have multiplicity one are called leaves. Our goal is to study the convergence of random real trees. To this end, it is of course necessary to have a notion of distance between two real trees. We will use the Gromov-Hausdorff distance between compact metric spaces, which has been introduced by Gromov (see e.g. \cite{G}) in view of geometric applications.

If $(E,\delta)$ is a metric space, we use the notation $\delta_{Haus}(K,K')$ for the usual Hausdorff metric between compact subsets of $E$:
\[
 \delta_{Haus}(K, K') = \inf\left\{\epsilon > 0 : K \subset U_{\epsilon}(K') \; \text{and}\; K' \subset U_{\epsilon}(K)\right\}, 
 \]
where $U_{\epsilon}(K) := \{x \in E:\; \delta(x,K) \le \epsilon \}$. Then, if $\rtree$ and $\rtree'$ are two rooted compact metric spaces, with respective roots $\rho$ and $\rho'$, we define the distance $d_{GH} (\rtree , \rtree')$ by 
\[
d_{GH}(\rtree ,\rtree') = \inf\left\{ \delta_{Haus}\left(\phi(\rtree),\phi'(\rtree')\right) \vee  \delta(\phi(\rho),\phi'(\rho'))\right\}
\]
where the infimum is over all choices of a metric space $(E,\delta)$ and all isometric embeddings $\phi: \rtree\longrightarrow E$ and $\phi':T' \longrightarrow E$ of $\rtree$ and $\rtree'$ into $(E,\delta)$. Two rooted compact metric spaces $\rtree_1$ and $\rtree_2$ are called equivalent if there is a root-preserving isometry that maps $\rtree_1$ onto $\rtree_2$. Obviously $d_{GH} (\rtree, \rtree')$ only depends on the equivalence classes of $\rtree$ and $\rtree'$. Then $d_{GH}$ defines a metric on the set of all equivalent classes of rooted compact metric spaces (see \cite{G, LGsurv, EPW}). We denote by $\rtspace$ the set of all (equivalence classes of) rooted real trees. 

\begin{thm}[See \cite{EPW}]
The metric space $(\rtspace,d_{GH})$ is complete and separable.
\end{thm}

\begin{defn}
Weighted real trees are compact real trees equipped with a probability measure on its Borel $\sigma$-field (informally, the \textit{leaf distribution}). Let $\wrtspace$ denote the set of all weighted real trees. 
Continuum trees are weighted real trees such that the support of the leaf distribution is the entire tree.
\end{defn}

We lift the following definitions and results from \cite{EW}. Recall the Prokhorov metric between two probability measures $\mu$ and $\nu$ on a metric space $(X, d)$ with the corresponding collection of closed sets denoted by $\mathfrak{C}$:
\[
d_P(\mu, \nu) := \inf\left\{ \epsilon > 0:\; \mu(C) \le \nu \left( C^\epsilon  \right) + \epsilon\; \text{for all}\; C \in \mathfrak{C} \right\},
\]
where $C^\epsilon:=\{ x\in X:\; d(x,C) < \epsilon  \}$.

An $\epsilon$-(distorted) isometry between two metric spaces $(X, d_X)$ and $(Y, d_Y)$ is a (possibly nonmeasurable) map $f:X \longrightarrow Y$ such that
\[
\dis(f):= \sup \left\{  \abs{d_X(x_1, x_2) - d_Y(f(x_1), f(x_2))}; x_1, x_2 \in X  \right\} \le \epsilon. 
\]
The following lemma is important in its own right. 

\begin{lemma}[\cite{EW}, Lemma 2.1]
Let $(\rtree_1, \dtree_1)$ and $(\rtree_2, \dtree_2)$ be two compact real trees such that $d_{GH}((\rtree_1, \dtree_1), (\rtree_2, \dtree_2)) < \epsilon$ for some $\epsilon > 0$. Then, there exists a measurable $3\epsilon$-isometry from $\rtree_1$ to $\rtree_2$.
\end{lemma}

We now define the weighted Gromov-Hausdorff distance between two metric spaces with probability measures $(X, d_X, \nu_X)$ and $(Y, d_Y, \nu_Y)$. For $\epsilon > 0$, set
\[
F_{X,Y}^\epsilon := \left\{ \text{measurable $\epsilon$-isometries from $X$ to $Y$}   \right\}.
\]
Define
\[
\begin{split}
\Delta^*_{GH}&:= \inf\left\{ \epsilon > 0:\; \exists\; f \in F^\epsilon_{X,Y}, g \in F^\epsilon_{Y,X}\; \text{such that}\right.\\
& \left. d_P\left( \nu_X\circ f^{-1}, \nu_Y  \right)\le \epsilon,\; d_P\left( \nu_X, \nu_Y\circ g^{-1}  \right)  \le \epsilon  \right\}.
\end{split}
\]

$\Delta^*$ satisfies all the properties of a metric except the triangle inequality. To rectify this, we define
\[
d^*_{GH}\left( X,Y \right) := \inf\left\{ \sum_{i=1}^{n-1} \Delta^*_{GH}\left( Z_i, Z_{i+1}  \right)^{1/4}  \right\},
\]
where the infimum is taken over all finite sequences of compact, weighted real trees $Z_1, Z_2, \ldots, Z_n$ with $Z_1=X$ and $Z_n=Y$.

\begin{lemma}[\cite{EW}, Lemma 2.3]
The function $d^*_{GH}$ is a metric on weighted real trees. Moreover,
\[
\frac{1}{2} \Delta^*_{GH} (X,Y)^{1/4} \le d^*_{GH}(X,Y) \le \Delta^*_{GH} \left( X,Y \right)^{1/4}.
\]
\end{lemma}

\begin{thm}[\cite{EW}, Theorem 2.5] The metric space $(\wrtspace, d^*_{GH})$ is complete and separable. 
\end{thm}

The above theorem shows that the space of continuum trees is suitable for doing probability. 

\subsection{Chronological trees}

Chronological trees are particular instances of real trees. We closely follow the notations and description provided in the article by Lambert \cite{L}. Also see the book by Jagers \cite[Chapter 6]{J} where they are treated as special cases of a general branching process. Informally, this a model of a population where each individual exists during her \textit{lifespan}, and produces offsprings during the time of her existence. Splitting trees are random chronological trees where we are going to assume that lifetime of individuals are iid from some suitable law, and offsprings are produced at a constant rate.

Somewhat more formally, such trees can be considered as the set of edges of some discrete tree embedded in the plane, where each edge length is a {lifespan}. Each individual of the underlying discrete tree possesses a \textit{birth level} $\alpha$ and a \textit{death level} $\omega$, both nonnegative real numbers such that $\alpha < \omega$, and a nonnegative number of offsprings with distinct birth times belonging to the interval $(\alpha, \omega)$. We think of a chronological tree as the set of all so-called \textit{existence points} of individuals of the discrete tree. See Figure 2, reproduced from \cite{L}.

More rigorously, consider the space of discrete trees in \eqref{whatiscalu} and let 
\eq\label{whatisbbu}
\chtree= \discretetreesp \times [0,\infty)
\en
and set $\rho=(0,0)\in \chtree$. We let $\proj_1$ and $\proj_2$ respectively to be the canonical projections of $\chtree$ on $\discretetreesp$ and $[0,\infty)$. The first projection of any subset $\tree\subseteq \chtree$ will be denoted by $\discretetree$,
\[
\discretetree:=\proj_1\left(  \tree \right)=\left\{ u\in \tree:\;\exists \sigma\ge 0, \; (u,\sigma) \in \tree   \right\}.
\]

A chronological tree $\tree$ is a subset of $\chtree$ such that:
\begin{enumerate}
\item[(i)] $\rho \in \tree$ and is called the root. 
\item[(ii)] $\discretetree$ is a discrete tree.
\item[(iii)] For any $u\in \discretetree$, there are $0\le \alpha(u) < \omega(u) \le \infty$ such that $(u,\sigma)\in \tree$ if and only if $\sigma \in (\alpha(u), \omega(u)]$.
\item[(iv)] For any $u\in \discretetree$ and $j \in \mathbb{N}$ such that $uj \in \discretetree$, $\alpha(uj) \in (\alpha(u), \omega(u))$.
\item[(v)] For any $u\in \discretetree$ and $i,j\in \mathbb{N}$ such that $ui, uj \in \discretetree$,
\[
i\neq j \Rightarrow \alpha(ui) \neq \alpha(uj).
\]
\end{enumerate}

\begin{figure}[t]\label{fig_chronology}
\centering
\includegraphics[width=5in, height=2.5in]{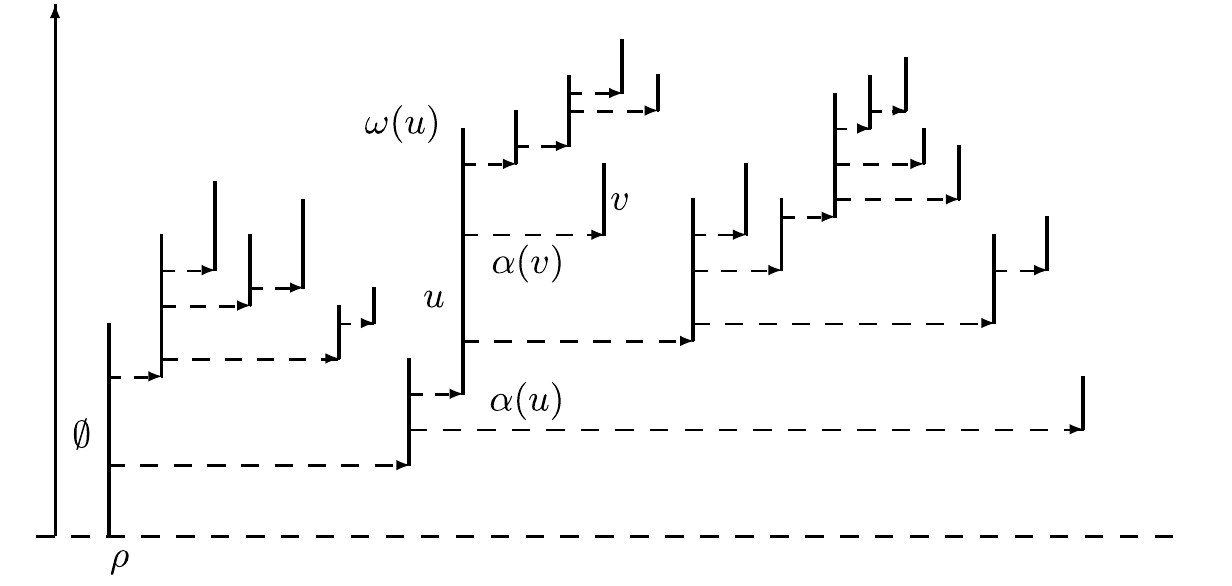}
\caption{A chronological tree (reproduced from \cite[p.~353]{L})}
\end{figure}

For any $u\in \discretetree$, $\alpha(u)$ is the birth level of $u$, $\omega(u)$ is the death level and we denote by $\zeta(u)$ its lifespan $\omega(u) - \alpha(u)$. 

Note that, condition (iii) above implies that if $\tree$ is not reduced to $\rho$, then the root $0$ of $\discretetree$ has a positive lifespan $\zeta(0)$. 

\bigskip

We now discuss planar embedding of a chronological tree. This is merely formalizing the way we have drawn Figure 2. The most important rule is: \textit{edges always grow to the right}. For any $x\in \tree$, we denote by $\theta(x)$ the descendants of $x$, that is, the subset of $\tree$ containing all $z\in \tree$ such that $x \prec z$. The descendants of $x$ can be split int its \textit{l(eft)-descendants} $\theta_l(x)$ and \textit{r(ight)-descendants} $\theta_r(x)$. Their definitions are as follows: if $x$ is not a branch point, $\theta_l(x)=\theta(x)$ and $\theta_r(x)$ is the empty set $\emptyset$; if $x=(u,\sigma)$ is a branch point, then $\sigma=\alpha(uj)$ for some integer $0\le j \le K_u$, and
\[
\theta_l(x) := \bigcup_{\epsilon > 0} \theta(u, \sigma+\epsilon) \quad \text{and}\quad \theta_r(x) := \{x\} \cup \bigcup_{\epsilon > 0} \theta(uj, \sigma+\epsilon).
\]
Note that we have adopted the convention of putting $x$ itself as its own $r$-descendant. 
\bigskip

Finally we need to define a proper notion of depth-first ordering for chronological trees. We define the following linear order following \cite[p.~356]{L}. 
\medskip

\noindent\textit{Linear order ``$\le$''.} Let $x,y \in \tree$. Unless $x\wedge y \in \{ x,y\}$, either $y \in \theta_r(x\wedge y)$ or $x\in \theta_r(x\wedge y)$. Define a total order by
\[
x \le y \quad \Leftrightarrow \quad y \prec x,\; \text{or}\; x \in \theta_l(x\wedge y) \quad \Leftrightarrow \quad y \prec x,\;\text{or}\; y \in \theta_r(x\wedge y).
\]
Note that the genealogical order and the linear order are reversed. 
\bigskip

There is another, more classical, way of describing a chronological tree. This is called the \textit{age process}. The following description is for a finite chronological tree. 

\begin{defn} (The Age process)
For every level $\sigma \in [0,\infty)$ consider the collection of existence points $\{ (u_{i},\sigma),\; 0 \le i \le m_{\sigma}\} \subseteq \tree$, where $m_\sigma$ is some nonnegative integer depending on $\sigma$. We assume that $u_{i}$'s are ordered in the linear order $\le$ described above.
This results in a point process 
\[
\age(\sigma)=\left\{  \left( k, \sigma - \alpha(u_k) \right) ,\; k=0,1,\ldots,m \right\}.
\]
In words, $\age(\sigma)$ keeps track of all the existence points at level $\sigma$ and their current ages. The process $\{\age(\sigma),\; \sigma \ge 0 \}$ is called the age-process. 
\end{defn}

A splitting tree is a random chronological tree characterized by a $\sigma$-finite measure $\Lambda$ on $(0,\infty]$ called the lifespan measure, satisfying
\[
\int_0^\infty (r\wedge 1) \Lambda(dr) < \infty.
\]

Given a lifespan measure, let $\splitlaw_\chi$ denote the law of a splitting tree starting with one ancestor individual $0$ having deterministic lifetime $(0,\chi]$, where $\chi$ is only allowed to be $\infty$ when $\Lambda(\{+\infty\}) > 0$. This, informally, means that for each individual $v$ of the tree, conditional on $\alpha(v)$ and $\omega(v)$, the pairs $(\alpha(vi), \zeta(vi))_{i\ge 1}$ made of the birth levels and lifespans of $v$'s offspring are the atoms of a Poisson measure on $(\alpha(v), \omega(v))\times (0,\infty]$ with intensity measure $\leb\otimes \Lambda$. Here $\leb$ refers to the Lebesgue measure. In addition, conditionally on this Poisson measure, descending subtrees issued from these offsprings are independent.  

This can also be expressed in terms of grafting (please see \cite{L} for definitions). Suppose $g(\tree', \tree, x,i)$ is the tree obtained by grafting $\tree'$ on $\tree$ at $x$, as descending from $\proj_1(x)i$.  Let $(\alpha_i, \zeta_i)_{i\ge 1}$ be the atoms of a Poisson measure on $(0,\chi) \times (0,\infty]$ with intensity measure $\leb\otimes \Lambda$. Then $\splitlaw_\cdot$ is the unique family of probability measures on chronological trees satisfying 
\[
\tree \stackrel{\text{law}}{=} \bigcup_{n\ge 1} g\left( \tree_n, 0\times (0,\chi), (0, \alpha_n), n  \right),
\]
where conditionally on the Poisson measure, the $\tree_n$'s are independent splitting trees with law $\splitlaw_{\zeta_n}$, and $\tree$ has law $\splitlaw_\zeta$.

Finally, set
\[
m := \int_0^\infty r \Lambda(dr),
\]
and we say that $\tree$ is subcritical, critical, or supercritical, according to whether $m< 1,=1$, or $> 1$.

We now define (as in \cite{L}) the \textit{jumping chronological contour process} (JCCP) for a finite splitting tree. Informally, when the tree is finite, it is the distance from the root of a path on the chronological tree that starts from the root and traverses the tree in a depth first order in the following manner: at any branchpoint, the path jumps up to the end the lifespan of the individual born at that branchpoint and then decreases linearly (at speed $-1$) along the lifetime of this individual until it encounters a new birth event. When the contour process ends its visit of an individual, its value is thus the birth level of this individual. It then continues to visit its mother's lifetime at the level where (when) it had been interrupted. The basic theorem we require is that the JCCP for a splitting process is a L\'evy process with a L\'evy measure $\Lambda$. This follows since the jumps are iid and appears at rate one, according to our construction of the splitting tree. A figure of a JCCP for a finite chronological tree is shown in Figure 3. For more pictures and details, please see \cite{L}.

\begin{figure}[t]
\centering
\includegraphics[width=4in, height=2 in]{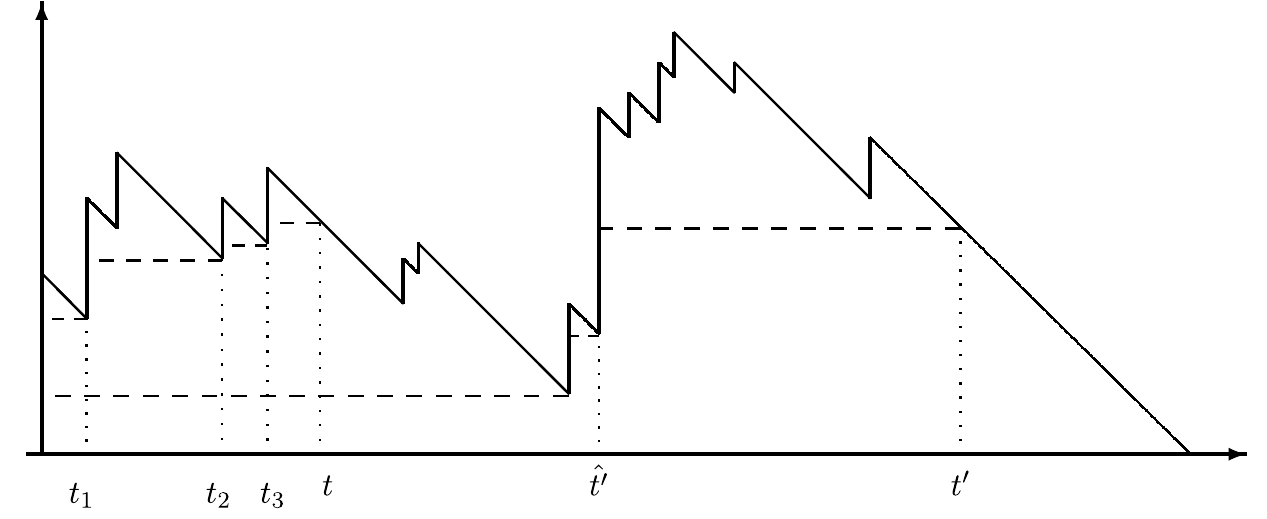}
\caption{A typical JCCP (reproduced from \cite[p.~370]{L})}
\label{fig_jccp}
\end{figure}

We now present the formal definition closely following \cite{L}. We start with the exploration process. From now on, $\tree$ will denote a finite chronological tree with finite total length $\ell:=\lambda(\tree)$. The real interval $[0,\ell]$ is equipped with its Borel $\sigma$-algebra and the Lebesgue measure. For any $x\in \tree$, set
\[
\mathbb{S}(x) := \{ y \in \tree:\; y \le x \}.
\]
It has been shown in \cite{L} that the map $\varphi:\tree \rightarrow [0,\ell]$ given by
\[
\varphi(x):= \lambda(\mathbb{S}(x)),\qquad x\in \tree
\]
is the unique measurable bijection that preserves the Lebesgue measure on both metric spaces and the order. Note that 
\[
\varphi(0, \omega(0)) =0, \quad \text{and}\quad \varphi(\rho) =\ell.
\]  

\begin{defn}
The right-continuous inverse $(\varphi^{-1}(t);\; t\in [0,\ell])$ is called the exploration process. The process given by the second projection, i.e., 
\[
X_t:=\proj_2 \circ \varphi^{-1}(t), \quad t\in [0,\ell]
\]
is called the JCCP, or, the jumping chronological contour process.
\end{defn}

The main result we need now follows.

\begin{thm}[JCCP for splitting trees, Theorem 3.3 and 4.3 in \cite{L}]\label{JCCP_levy} The JCCP for a chronological tree is c\`adl\`ag and given by the expression
\[
X_t = -t + \sum_{v:\varphi(v,\omega(v))\le t} \zeta(v),\quad 0\le t \le \ell.
\]
Consider the law of a splitting tree $\splitlaw_\chi$ which goes extinct almost surely. Then $X$ has the law of a spectrally positive L\'evy process of finite variation, with a Laplace exponent
\[
\varphi(\theta) = \theta - \int_0^\infty \left( 1 - e^{-\theta r}  \right)\Lambda(dr), \quad \theta\ge 0,
\]
killed upon hitting zero. 
\end{thm}

For definitions of spectrally positive (or negative) L\'evy processes, please see the classic text by Bertoin \cite[Chapter VII]{B} from where we will use all the results we need in the following text. Note that when the tree is finite the JCCP is nothing but a compound Poisson process with drift.

\section{Preliminaries on point processes and subordinators} Point processes can be, of course, treated under the rubric of the theory of random measures which is going to be our approach here. A standard reference to the theory is the book by Daley and Vere-Jones \cite{DVJ}, Chapter 7. 

Let $(X, \rho)$ be an arbitrary complete separable metric space equipped with its Borel $\sigma$-algebra $\borel_X$. A Borel measure $\mu$ on $X$ is boundedly finite if $\mu(A) < \infty$ for every bounded Borel set $A$. Let $\mx_X$ denote the set of all boundedly finite integer-valued Borel measures on $X$. Then there is a natural metric on $\mx_X$ that turns it to a complete separable metric space. The corresponding Borel $\sigma$-algebra is the smallest $\sigma$-algebra with respect to which the mapping $\mu \mapsto \mu(A)$ is measurable for all $A\in \borel_X$. Finally we include an extra point $\emptyset$ in $\mx_X$ to represent the empty measure.  

\begin{defn} The following definitions and results are standard and can be found in, for example, \cite{DVJ}. 
\begin{enumerate}
\item[(i)] A random point process $N$ on $X$ is a probability measure on $\mx_X$. 
The finite-dimensional distribution refer to the collection of laws of finite dimensional random vectors
\[
\left( N(A_1), \ldots, N(A_n)  \right)
\]
where $n \in \mathbb{N}$ and $A_1, \ldots, A_n$ are disjoint bounded Borel subsets of $X$.
The law of $N$ is characterized by its finite-dimensional distributions.  
\bigskip

\item[(ii)] Let $\nu$ be a $\sigma$-finite nonnegative measure on $X$. A Poisson point process (PPP) with an intensity measure $\nu$ is a random point process on $X$ whose finite dimensional distributions are given by independent Poisson random variables with means $(\nu(A_1), \ldots, \nu(A_n))$
\bigskip

\item[(ii)] Recall that a Radon measure on a locally compact space is one that puts finite mass on compact sets. 
 Consider a sequence of Radon measures $\{\mu_n\}$ on a locally compact space. We say that the sequence converges to a vaguely to a Radon measure $\mu$ if 
\[
\lim_{n\rightarrow \infty} \int f d \mu_n = \int f d\mu
\]
for all continuous $f$ vanishing outside a compact set. It follows that if the underlying space is a complete separable metric space then a family of Radon measures $\{ \mu_{\alpha} \}$ is relatively compact in the vague topology if and only if $\{ \mu_\alpha(A) \}$ is bounded for every bounded Borel set $A$.
\bigskip

\item[(iv)] On a complete separable metric space weak convergence of random measures (including point processes) is quivalent to the weak convergence of its finite dimensional distributions to the appropriate limit. 
\end{enumerate}
\end{defn}

We now state a classical result about convergence to Poisson point processes. 

\begin{lemma}\label{PPPconvergence} Let $\{\mu_n,\; n\in \mathbb{N} \}$ be a sequence of Borel probability measures on $X$. Suppose $\nu$ is a $\sigma$-finite measure on $X$ such that $n \mu_n$ converges vaguely to $\nu$.

For every $n\in \mathbb{N}$, consider an iid sequence of random elements $Y_{1,n}, Y_{2,n}, Y_{3,n}, \ldots$ with law $\mu_n$. Define a point process on $\rr^+ \times X$ by
\[
Z_n = \sum_{k=1}^\infty \delta_{({k}/{n}, Y_{k,n})},  
\]
where $\delta_{x}$ puts a unit mass at $x$. Then the sequence $\{ Z_n, \; n \in \mathbb{N} \}$ converges in law to a PPP on $\rr^+ \times X$ with intensity measure $\leb \times \nu$, where $\leb$ stands for the Lebesgue measure on the positive half-line.
 \end{lemma}

\begin{proof} This is a standard result. Hence we only outline the proof. It suffices to show (see \cite[Proposition 9.1.VII, page 275]{DVJ}) that for any bounded continuous function $f: \rr^+\times X \rightarrow \rr$ that vanishes outside a bounded set, we have the following limit
\[
\lim_{n\rightarrow \infty} E \exp \left[ - \sum_{k=1}^\infty f\left( k/n, X_{k,n} \right)  \right] = \exp\left[  - \int_0^\infty \int_{ X} \left( 1 - e^{-f(y,x)} \right) \nu(dx)dy\right].
\]
Suppose, without loss of generality, $f(y,\cdot)$ vanishes whenever $y > 1$. Then, the above limit follows by first proving it by approximating $k/n$ by $U_{(k)}$, the $k$th smallest order statistics in an iid sample $U_1, \ldots, U_n$ from Uni$[0,1]$, for which the limit is immediate, and then showing that the approximation is asymptotically negligible.  
\end{proof}

We will be working with explicit densities and the following lemma will come in handy. The locally compact space mentioned in the next lemma will be often in our case the positive quadrant in $\rr^d$ without the origin. 

\begin{lemma}\label{densityvague}
Let $\{f_n,\; n\ge 0\}$ be a sequence of positive continuous functions in a locally compact space. Suppose 
\begin{enumerate}
\item[(i)] $\lim_{n\rightarrow \infty} f_n(x) = f_0(x)$ for all $x$.
\item[(ii)] The family $\{f_n, \; n\ge 0\}$ is uniformly locally bounded. 
\end{enumerate} 
Let $\mu_n$ be the $\sigma$-finite measure whose density is given by $f_n$ with respect to some base Radon measure. Then $\mu_n$ converges vaguely to $\mu_0$ as $n$ tends to infinity.    
\end{lemma}

\begin{proof}
It is known (see \cite{DVJ}) that it suffices to show that if $A$ is a compact set then $\mu_n(A)$ converges to $\mu_0(A)$, as $n$ tends to infinity. But $f_n$ is bounded on $A$ and hence converges in $\mathbb{L}^1$ on $A$. This completes the proof. 
\end{proof}

\subsection{Poisson Additive Point Processes (PAPP)} 

We introduce the concept of an Poisson Additive Point Process. This is a Markov process of point processes on $\rr^+ \times X$, where $X$ is some complete separable metric space. We need the following definitions.
 
Let $\{\theta_s,\; s\ge0\}$ be the right shift operators on the first coordinate of $\rr^+\times X$, i.e., 
\[
\theta_s(x,y)=(x+s, y), \quad \text{for all}\; s,x \ge 0,\; y \in X.
\]
We need the following definitions.

\begin{defn}
Let $\left(A_t,\; t\in I \right)$ be a jointly-defined family of point processes on $(0,\infty)\times X$ indexed by some subset $I$ of $\rr^+$. Let $A_t(\cdot, X)$ denote the marginal measure of the first coordinate. A concatenation of the family $(A_t)$, denoted by $\oplus_{t\in I} A_t$, is an operation which is permitted only when almost surely
\begin{enumerate}
\item[(i)] each $A_t(\cdot, X)$ is supported on a compact interval $[0,\tau_t]$,
\item[(ii)] either each $A(\cdot, X)$ has no atom at $0$, or each $A(\cdot, X)$ has no atom on $\tau_t$
\item[(iii)] and, $\sum_{s\le t} \tau_s < \infty$ for every $t > 0$.
\end{enumerate}
When it is permitted, it is defined to be another point process $Z$ on $\rr^+\times X$ such that for any bounded Borel subsets $B\subseteq \rr^+ \times X$ we have 
\[
Z\left( B \right) = \sum_t A_{t} \circ \theta^{-1}_{\sum_{s\le t} \tau_s} \left( B \right).
\]
\end{defn}

Intuitively one can think of $(A_t)$ as a family of marked point processes on $\rr^+$ with marks in $X$. The above definition means that we are simply concatenating the distribution functions of $A_t(\cdot, X)$ while retaining the \textit{marks} on atoms.

\begin{defn}[\textbf{PAPP}] A Poisson additive point process, or PAPP, is a homogenous Markov process $\{N_t,\; t\ge 0\}$ on the space of point processes on $\rr^+\times X$ with the following property. For every $t\ge 0$, there is a Poisson point process $Z$ on $\rr^+ \times X \times \mx_{\rr^+\times X}$, whose law depends on $t$, with the following property. Let $I$ denote the set of atoms of $Z(\cdot, X, \mx_{\rr^+\times X})$. Consider the two derived point processes $(Y,W)$ defined by 
\[
Y= Z\left( \cdot, \cdot, \mx_{\rr^+\times X}  \right), \quad \text{and}\quad W=\oplus_{t\in I} Z\left( t, X, \cdot \right).
\]
Then $(Y,W)$ is a pair of jointly defined point processes on $\rr^+ \times X$ and the regular conditional distribution of $W$ given $Y$ is the transition kernel at time $t$ for the Markov process $N$.
\end{defn}

The previous definition is required to formalize the evolution of the age process of a splitting tree. Hence they have a natural relation with spectrally positive L\'evy processes. However, as we will see later, the PAPP is a more general concept and can exist even when there is no L\'evy process corresponding to it.

We need a few more standard operations on point processes. 

\begin{defn}
Let $(B_t,\; t \in I)$ be a countable family of point processes seen as random measures on a complete separable metric space $(X,\rho)$. A superposition of the family is a random measure $B$ such that for all Borel subset $A \subseteq X$ we get 
\[
B(A) := \sum_{t \in I} B_t(A).
\]
I.e., the set of atoms of $B$ is the union of the atoms in the family $(B_t, \; t \in I)$.
\end{defn}

\begin{defn}
Let $A$ denote a point process $\rr^+ \times \rr^+$. A translation of $A$ by a vector $(x,y)\in \rr^2$ is a point process which is a translation of the measure corresponding to $A$. We will denote this point process by $A + (x,y)$.  
\end{defn}

The necessity of the following definition will not make sense until later in the text. 

\begin{defn}[\textbf{The restriction consistency property}]\label{resconsis}
A family of point processes $\{B_t, \; t\ge 0\}$ on the state space $\rr^+ \times \rr^+$ is said to satisfy the restriction consistency property if the following holds. Consider any pair of indices $t_0 < t_1$. Consider $B_{t_1}$ restricted to the subset $\rr^+ \times [t_1-t_0, \infty)$. Then, there is a subset $D \subseteq \rr^+$, such that the collection of atoms of the second coordinate under $B_{t_0}$ restricted to $D \times \rr^+$ is identical to the collection of atoms of the second coordinate of the restricted measure $B_{t_1}$ translated by $(0, -t_1+t_0)$.
\end{defn}

\bigskip

\subsection{The Stable(1/2) subordinator} A big role in this analysis is played by the Stable($1/2$) subordinator and its associated point process. Hence we briefly define the process and list some of its properties. A more detailed account can be found in the book by Bertoin \cite{B}.

\begin{defn} The Stable($1/2$) subordinator is an increasing L\'evy process $X$ taking values in $[0,\infty)$. It can be characterized by its Laplace exponent 
\[
\Phi(\lambda):=-\frac{1}{t}\log  E \left( \exp - X_t  \right)
\]
which is proportional to $\sqrt{\lambda}$. The jump distribution $\Pi(dx)$ is proportional to $\mu(dx):=x^{-3/2}dx$ on $(0,\infty)$. Thus its jumps are naturally associated with a Poisson point process on $(0,\infty) \times (0,\infty)$ with an intensity measure $\leb\times \mu$. We call this point process as the Stable($1/2$) point process. 
\end{defn}

\section{Streets, mailman, and the address protocol} 

A convenient space for embedding real trees is the Banach space $\ell^1$ of sequences that are absolutely summable. The following description goes back to Aldous' original paper \cite{A91}. 
\bigskip

\noindent(\textit{The $\ell^1$ representation of continuum trees}) Let $\{\mathbf{e_i}, \; i \in \mathbb{N}\}$ denote the standard basis of $\ell_1$ which is $1$ in the $i$th coordinate and zero elsewhere. Let $(L_n)$ denote a sequence of positive numbers. We inductively define a real tree as follows. Let $u_1$ be the origin and let
\[
\mathcal{T}_1 = u_1 + \mathbf{e_1} [0, L_1] := \left\{  u_1 + x\mathbf{e_1},\; 0\le x \le L_1 \right\}.
\]
Now we continue inductively. Suppose we have defined $\mathcal{T}_k$ and we select a point $u_{k+1}$ from $\mathcal{T}_k$ and set
\[
\mathcal{T}_{k+1} = \mathcal{T}_k \cup \left\{ u_{k+1} + \mathbf{e_{k+1}}[0, L_{k+1}]  \right\}.
\]
Let $\mathcal{T}$ be the closure in $\ell^1$ of the union $\cup_{n\ge 1} \mathcal{T}_n$ then $\mathcal{T}$ is a real tree with a length metric derived naturally from $\ell^1$.

To get a random real tree we randomly choose $(L_n)$ and $(u_n)$. For example, a beautiful result due to Aldous states that if $(L_n)$ are the interarrival times times of a Poisson process on $\rr^+$ with rate $tdt$ and $u_n$ is is chosen uniformly at random from the existing tree structure, the resulting real tree is the Brownian CRT.

\bigskip

The $\ell^1$ representation leads to the following coding of a continuum tree by a stochastic process. Let $\rrp$ denote the set of nonnegative numbers and let $\meas$ denote the space of Borel measures on $\rrp$. We assign an \textit{address} system to our tree which is formally the law of a (possibly random) discrete-time stochastic process whose state space is $\rrp$. We show below that every stochastic process that satisfies certain properties naturally codes a continuum tree, and conversely, for every continuum tree, one can find a (actually, infinitely many) stochastic process that codes the tree. We start with a definition.

\begin{figure}[t]\label{fig_tree}
\centering
\includegraphics[width=5.3in, height=3in]{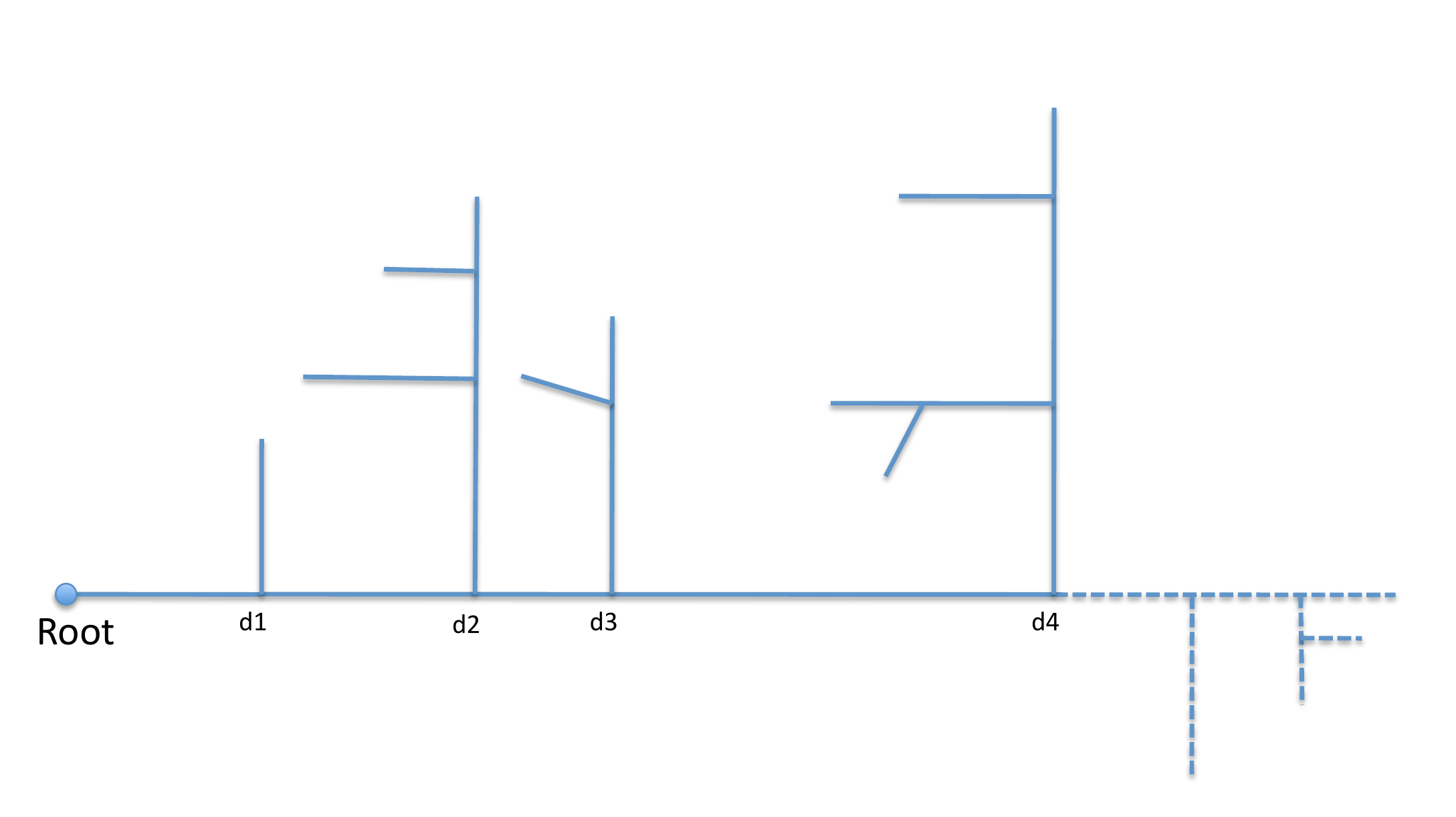}
\caption{The 1st street}
\label{imagestreet}
\end{figure}

\begin{defn}[\textbf{Streets}]\label{whatisastreet}
A street is a triplet $(\nu, I, R)$ where $I, R$ are nonnegative real numbers and $\nu$ is a finite positive discrete measure supported on $[0, I]$. Throughout the text we will refer to $I$ as the length of the street, $\nu$ as the measure on the street, and $R$ as the clock. The distribution function of $\nu$ will be denoted by $\nu$ itself. In particular, $\nu(a)$ refers to the $\nu$ measure of the set $[0,a]$. 
\end{defn}

\bigskip

\begin{defn}[\textbf{Mailman and the address protocol}]
Given a countable collection of streets $\mathbb{S}=\{ (\nu_n, I_n, R_n),\; n=1,2,\ldots \}$, one can define a stochastic process in the following way. A pair of real-valued stochastic processes $\{(A_n, \varsigma_n),\; n=1,2,3\ldots\}$ defined on a probability space $(\Omega,\mathcal{F},P)$ will be called a mailman on the streets $\mathbb{S}$ if it satisfies the following properties.
\begin{enumerate}
\item[(i)] Every $\varsigma_n$ is a Bernoulli random variable with
\[
P\left(  \varsigma_n=1 \right) =\frac{\nu_n(I_n)}{\nu_n(I_n) + R_n}.
\]
\item[(ii)] Given $\varsigma_n=1$, the law of $A_n$ is given by $\nu_n$ normalized to have mass one, while given $\varsigma_n=0$, we take $A_n=I_n$.
\item[(iii)] Moreover $\esssup\sum_{n=1}^\infty A_n < \infty$.
\end{enumerate}
Thus the sequence $(A_1, A_2, \ldots)$ is a random element in $\ell^1$. The law of the mailman, considering it as a probability measure on $\ell^1$, will be called the address protocol. Finally, the plural of mailman will be mailmen.
\end{defn}

The way a continuum tree is coded by the address protocol follows exactly the $\ell^1$ representation. Recall Aldous' definitions of a proper $k$-tree which is nothing but a rooted Cladogram with $k$ leaves whose edges have a real length. 

\begin{defn}[Proper $k$-trees]
Let $(\mathcal{R}(k);\; k\ge 1)$ be a family of random proper $k$-trees. For $j\le k$ let $(L_1^k, \ldots, L_j^k)$ be uniform random choice of $j$ distinct leaves of $\mathcal{R}(k)$. The family is consistent if, for each $1\le j \le k < \infty$, the reduced subtree generated by these leaves have the same law as $\mathcal{R}(j)$.
\end{defn}

\begin{defn}[Leaf-tight property] With the preceding notation, call the family $(\mathcal{R}(k),\; k\ge 1)$ leaf-tight if
\[
\min_{2\le j \le k} d\left( L_1^k, L_j^k \right) \stackrel{p}{\rightarrow} 0,\quad \text{as}\; k\rightarrow \infty,
\]
where $d$ is the natural graph distance on the real tree. 
\end{defn}

According a fundamental result by Aldous a continuum tree is uniquely specified by a consistent family of \text{proper $k$-trees} which satisfies the leaf-tight property. We state the results below (a bit sloppily in order to avoid brining in new definitions and formalism).

\begin{thm}\cite[Theorem 3]{A93} (i) Let $(\mathcal{R}(k),\; k\ge 1)$ be a consistent family of proper $k$-trees. Suppose the leaf-tight property holds. Then there exists a special continuum random tree $(\mathcal{S},\mu)$ with the following property. Let $(Z_i)$ be an exchangeable sequence directed by the random measure $\mu$. Then, for each $k$, the reduced subtree generated by $Z_1, \ldots, Z_k$ has law $\mathcal{R}(k)$.

(ii) Conversely, for every continuum random tree, the above mechanism defines a family of random trees $(\mathcal{R}(k))$ which is consistent and leaf-tight.
\end{thm}
\bigskip

The mailman defines a consistent family of proper $k$-trees in the following manner. For every fixed $k\ge 1$, let $(A(1), \varsigma(1)), (A(2), \varsigma(2)), \ldots,(A(k), \varsigma(k))$ be $k$ iid realizations of the mailman. Consider first $(A(1), \varsigma(1))$ and $(A(2), \varsigma(2))$. Define the stopping time
\[
m=\min\{ n\ge 0: \; A_n(1) \neq A_n(2)\quad \text{or}\quad \varsigma_n(1) \neq \varsigma_n(2)  \}.
\]
Then, we build a proper $3$-tree with a marked Root and leaves $1$ and $2$ according to the following recipe. Let $v$ denote the most recent ancestor of $1$ and $2$. The distance between the Root and the vertex $v$ is given by
\[
d(v,Root):=\min\left\{  \sum_{i=0}^m A_i(1),\; \sum_{i=0}^m A_i(2)   \right\}.
\] 
The distance between the Root and the two leaves are given by (respectively)
\[
\sum_{i=0}^\infty A_i(1),\qquad \text{and}\qquad \sum_{i=0}^\infty A_i(2).
\]

We can now proceed by induction. Suppose we have drawn a $(j+1)$-tree with a marked Root and labeled leaves $1,2,\ldots, j$ using $(A(1), \varsigma(1)), \ldots, (A(j), \varsigma(j))$. Consider $(A(j+1), \varsigma(j+1))$ and define
\[
m_{t,j+1} = \min\{n\ge 0:\; A_n(t) \neq A_n(j+1) \}\quad \text{or}\quad \varsigma_n(t) \neq \varsigma_n(j+1) ,\quad t=1,2,\ldots,j.
\]
Then we add a leaf labeled $(j+1)$ at a distance of $\sum_{n=0}^\infty A_n(j+1)$ from the Root such that the most recent ancestor between leaf $(j+1)$ and leaf $t$ is at a distance
\[
d(j+1,t):=\min\left\{   \sum_{i=0}^{m_{t,j+1}} A_i(j+1),\; \sum_{i=0}^{m_{t,j+1}} A_i(t)   \right\}
\]
from the Root. Note that, this amounts to selecting $t^*$ such that $m_{t^*,j+1}=\max_t m_{t,j+1}$, and attaching an edge at the most recent ancestor between leaf $(j+1)$ and $t^*$ of appropriate length.  

This gives rise to a family of random $k$-trees which is consistent by construction. Note that the distance between leaf $1$ and leaf $j$ is given by
\[
d(j,1):=\sum_{n=0}^\infty A_n(1) + \sum_{n=0}^\infty A_n(j) - 2\min\left\{   \sum_{n=0}^{m_{j,1}} A_n(j),\; \sum_{n=0}^{m_{j,1}} A_n(1)   \right\}
\]
Note that every $A_i(\cdot)$ is supported on countably many atoms. Thus, given the path of $(A(1), \varepsilon(1))$, the iid random variables $\max_{1\le j\le k} m_{j,1}$ tend to infinity in probability as $k$ tends to infinity. Moreover, since $\esssup \sum_n A_n(\cdot)$ is bounded, we get $\min_{1\le j\le k} d(j,1)$ goes to zero in probability. This proves that the family of random $k$-trees generated by the mailman is leaf-tight.  

Hence we get the following result. 

\begin{lemma}\label{mailmantree}
A countable collection of streets such that there is a mailman on the streets defines a unique law on the space of continuum trees.
\end{lemma}

The reader might be wondering why we kept the variables $\varsigma_\cdot$. This is due to a fundamental reasons for which we provide the following heuristic. The above construction is very similar to the string-of-beads representations in the recent papers by Pitman and Winkel \cite{PW} and  Haulk and Pitman \cite{HP}. These are all related to the Chinese Restaurant construction of a random tree (please see the above cited papers). Recall the Chinese restaurant process for creating a random permutations (please see Chapter 3 in the excellent lecture notes by Pitman \cite{CSP}). Any new customer can either choose one of the existing tables, or can create a new table on her own. As will become clear in the following text, this process of creating a new table is being taken care of by the variables $\varsigma$. For example, consider Figure \ref{imagestreet}. There are five subtrees growing on four internal vertices on the spine at a distance $d_1, d_2, d_3,$ and $d_4$ away from the root. One can construct a street by associating five positive numbers with these five subtrees $\alpha_1, \ldots, \alpha_5$. Then $R=\alpha_5$, $I=d_4$, and the measure $\nu$ given by $\mu(\{d_i\})=\alpha_i$, for $i=1,2,3,4$.

\section{Galton-Watson trees with emigration} We begin the discussion with some results on a standard Galton-Watson tree with emigration. Consider a continuous time binary branching process $\{Z_t, \; t\ge 0\}$ with rate two and an additional rate one of emigration. In other words, the generator for this process is given by
\[
\mathcal{L}f(x) = 2x\left[  f(x+1) + f(x-1) - 2 f(x)  \right] + \left[ f(x-1)  - f(x) \right], \qquad x\in \mathbb{N}\cup \{0\},
\]
with an absorbing state zero. We denote the family of such process, starting with an given number of individuals, by $\gw(-1)$. Let $P^r$ denote that the law of a $\gw(-1)$ process starting with $r$ individuals.

\begin{lemma}\label{emi_compu}
Consider the continuous time binary branching process with emigration rate $1$ starting with one individual. Let $\sigma_0$ be the first time the population hits zero. Then, the Laplace transform of $\sigma_0$, under $P^1$, is 
\[
\begin{split}
\psi(\theta)&=E^{1}\left(  e^{-\theta \sigma_0} \right)=1 - \theta + \frac{1}{\sqrt{2}}\theta^{3/2} e^{\theta/2} \int_{\theta/2}^\infty t^{-1/2} e^{-t}dt. 
\end{split}
\]
In particular, the mean of $\sigma_0$ under $P^1$ is one. 
\end{lemma}

\begin{proof} We will use a variation of the {generating function technique}. For any nonnegative integer $r$, define
\[
\sigma_r:= \inf \left\{ t\ge 0:\; Z_t=r    \right\},
\]
where the infimum of an empty set is infinity.

Let $U_r(t)$ denote the density of $\sigma_r$ at time $t$ when the chain starts from $r+1$, and let $V_r(t)$ denote the probability
\[
V_r(t) = P^{r}\left( X_t=r, \; \text{and}\; X_s \ge r, \; \text{for all}\; s\le t  \right). 
\]

Now, for any $r\ge 1$, that event $\{ X_t=r,\; X_s \ge r\; \text{for all}\; s\le t\}$ can be decomposed into disjoint events where (i) either the chain stays at $r$ throughout, (ii) or, for any choice of $(u,s)$ such that $u+s\le t$, moves one step up to $r+1$ at time $u$, returns to $r$ at time $s$ for the first time, and stays $\ge r$ for the rest of the duration along with $X_t=r$. By the strong Markov property, one gets
\eq\label{Vrrec1}
V_r(t) = e^{-(4r+1)t} + \int_0^t 2r e^{-(4r+1)u} \int_0^{t-u} U_r(s) V_{r}(t-s-u) ds.
\en

Now, for a positive number $\theta$, define the transforms
\[
\begin{split}
F_r(\theta) &= \int_0^\infty U_r(t) e^{-\theta t} dt, \quad H_r(\theta) = \int_0^\infty  V_r(t) e^{-\theta t} dt,\\
G_r(\theta) &=\int_0^\infty e^{-(4r+1)t} e^{-\theta t}dt = \frac{1}{4r+1 + \theta}.
\end{split}
\]

Then, by Fubini-Tonelli, we get
\[
\begin{split}
2rF_rH_rG_r &= 2r\int_0^\infty  \left[ \int_{u+s+w=t} e^{-(4r+1)u} U_r(s) V_r(w)  \right] e^{-\theta t}dt\\
&= \int_0^\infty \left[  V_r(t) - e^{-(4r+1)t}  \right] e^{-\theta t} dt, \quad \text{by \eqref{Vrrec1}},\\
&= H_r(\theta) - G_r(\theta).
\end{split}
\]
Rearranging terms from above, we get
\[
H_r(\theta) = \frac{G_r(\theta)}{1 - 2 r F_r(\theta) G_r(\theta)} = \frac{1}{(4r+1+\theta) - 2r F_r(\theta)}, \quad r\ge 1.
\]

Now, for any $r\ge 0$, one can do a last passage decomposition to the event $\{ \sigma_r=t \}$, given $\{X_0=r+1\}$, by  considering the event that the chain remains $\ge r+1$ during the time interval $[0,t)$, is at $r+1$ at time $t-$, and the next jump it makes is exactly at time $t$ when it lands on the value $r$. More rigorously, for $\epsilon \approx 0$, we get
\eq\label{last_passage}
\begin{split}
\int_{t}^{t+\epsilon} &U_r(s)ds = P^{r+1}\left( t < \sigma_r \le t+\epsilon   \right) \\
&=\sum_{k=1}^\infty P^{r+1} \left(  X_t  = r+k, \; X_s \ge r+1,\; 0\le s\le t \right) P^{r+k}\left(  \sigma_r \le \epsilon\right)\\
&= V_{r+1}(t)P^{r+1}\left(  \sigma_r \le \epsilon \right) + o\left( \epsilon^2  \right).
\end{split}
\en
By dividing the above expression by $\epsilon$ and taking limit as $\epsilon \rightarrow 0$, we get 
\[
U_r(t)= (2r+3)V_{r+1}(t).
\]

Hence,
\[
\begin{split}
F_r(\theta) &= (2r+3) H_{r+1}(\theta)= \frac{2r+3}{(4r+5+\theta) - 2(r+1) F_{r+1}(\theta) }\\
&=\frac{r+3/2}{(2r+5/2+\theta/2)- (r+1)F_{r+1}(\theta)}.
\end{split}
\]

By recursion of the the above identity we arrive at the following continued function expansion:
\eq\label{contexp}
\begin{split}
F_0(\theta)& = \frac{3/2\mid }{\mid \theta/2+5/2} - \frac{5/2\cdot 1\mid}{\mid \theta/2+9/2} - \frac{7/2\cdot 2\mid}{\mid \theta/2+13/2} - \frac{9/2\cdot 3\mid}{\mid \theta/2 + 17/2} -\cdot\cdot\cdot
\end{split}
\en

Now, it is known (see \cite[p.~278, eqn. 3.3.3]{LW}) that the (complementary) incomplete Gamma function
\[
\Gamma(a,z):= \int_z^\infty e^{-t} t^{a-1} dt  
\]
satisfies the following continued fraction expansion
\[
e^z z^{-a} \Gamma(a,z)= \frac{1\mid}{\mid1+z-a} - \frac{1(1-a)\mid }{\mid 3+z-a} - \frac{2(2-a)\mid}{\mid 5+z-a} - \frac{3(3-a)\mid}{\mid 7+z-a} - \cdot\cdot\cdot, \qquad z > 0, \quad a\in \mathbb{C}.
\]

Putting $a=-3/2$ and $z=\theta/2$ in the above expression and comparing it with \eqref{contexp} we immediately get 
\eq\label{laplaceform1}
F_0(\theta) = \frac{3}{2}e^{\theta/2} \left(\frac{\theta}{2}\right)^{3/2}\Gamma(-3/2,\theta/2)=\frac{3}{2\sqrt{8}} e^{\theta/2}\theta^{3/2}\Gamma(-3/2,\theta/2).
\en
 
We now apply integration by parts to the incomplete Gamma function for $a \neq 0$:
\[
\Gamma(a,\theta) =- \frac{e^{-\theta}\theta^a}{a} + \frac{1}{a}\Gamma(a+1,\theta).
\]
Thus
\[
\begin{split}
\Gamma(-3/2,\theta/2)&=\frac{2}{3}e^{-\theta/2}\left(\frac{\theta}{2} \right)^{-3/2}-\frac{2}{3}\Gamma(-1/2,\theta/2)\\
&=\frac{2}{3}e^{-\theta/2}\left(\frac{\theta}{2} \right)^{-3/2}-\frac{2}{3}\left[ 2e^{-\theta/2}\left(\frac{\theta}{2}\right)^{-1/2} -2\Gamma(1/2,\theta/2) \right]\\
&=\frac{2\sqrt{8}}{3} e^{-\theta/2}\theta^{-3/2} -\frac{4\sqrt{2}}{3} e^{-\theta/2}\theta^{-1/2} +\frac{4}{3}\Gamma(1/2,\theta/2).
\end{split}
\]

Since $\psi(\theta)=F_0(\theta)$, we get
\[
\begin{split}
\psi(\theta) &= \frac{3}{2\sqrt{8}}\theta^{3/2}e^{\theta/2}\left[ \frac{2\sqrt{8}}{3} e^{-\theta/2}\theta^{-3/2} -\frac{4\sqrt{2}}{3} e^{-\theta/2}\theta^{-1/2} +\frac{4}{3}\Gamma(1/2,\theta/2) \right]\\
&= 1 - \theta + \frac{1}{\sqrt{2}}\theta^{3/2} e^{\theta/2} \int_{\theta/2}^\infty t^{-1/2} e^{-t}dt.  
\end{split}
\]
This completes the proof. 
\end{proof}

\begin{lemma}\label{sigma_density} The density of $2\sigma_0$ under $P^1$ is given by the function
\eq\label{densityform}
h(s) = \frac{3}{2}(1+s)^{-5/2}, \qquad s \ge 0.
\en
Hence if we let $L(u)=P^1(\sigma_0 \le u)$, then $\lbar(u):=1-L(u)=(1+2u)^{-3/2}$. 
\end{lemma}

\begin{proof} We simply invert the Laplace transform of $\sigma_0$ as obtained in the last lemma. We work with $\psi(2\theta)$:
\[
E\left(  e^{-\theta 2\sigma_0}\right)=\psi(2\theta)=\frac{3}{2}e^\theta \theta^{3/2}\Gamma(-3/2,\theta).
\]

By applying scale and shift transforms to \eqref{laplaceform1} we get
\[
\begin{split}
\psi(\theta) &= \frac{3}{2} e^{\theta}\theta^{3/2}\int_{\theta}^\infty e^{-t}t^{-5/2}dt = \frac{3}{2} e^\theta\int_{1}^{\infty} e^{-\theta u} u^{-5/2}du,\qquad t=\theta u\\
&= \frac{3}{2} \int_1^\infty e^{-\theta(u-1)}u^{-5/2}du = \frac{3}{2}\int_0^\infty e^{-\theta s} \frac{1}{(1+s)^{5/2}}ds, \qquad s=u-1.
\end{split}
\] 
By the uniqueness of the Laplace transform this proves the formula for the density. The expression for $L(u)$ follows easily. 
\end{proof}

We also need an idea about the size of the tree conditioned to have survived for long. We start with a natural definition which will be useful later. 

\begin{defn}\label{defininitionage} Suppose that the GW branching processes started with one individual as time $\alpha \in \rr$ and survives until time $\alpha + \sigma_0$. The age of the branching process is defined for times $t \in [\alpha, \alpha+\sigma_0]$, the age at time $t$ being $t - \alpha$.
\end{defn}

\begin{lemma}\label{expecsigma}
Consider the process $(Z_t, \; t\ge 0)$, a $\gw(-1)$ starting with one individual. Then 
\[
E\left(  Z_t \mid \sigma_0 > t \right) = 1+2t.\qquad %Var\left( Z_t \mid \sigma_0 > t \right)=(1+2t)^{3/2}-(1+2t).
\] 
In particular $\lim_{n\rightarrow \infty} E\left( Z_{nt}/n \mid \sigma_0 > nt \right) = 2$. 

\end{lemma}

\begin{proof} The process $Z_t + t$ is a martingale as the function $u(x,t)=x+t$ satisfies
\[
\frac{\partial u}{\partial t} + \mathcal{L} u \equiv 0.
\]
We apply the Optional Sampling Theorem at the stopping time $t\wedge \sigma_0$ to get
\[
E\left( Z_{t\wedge \sigma_0}  \right) = 1- E\left( t\wedge \sigma_0 \right).
\] 
The right side of the above equality can be evaluated explicitly thanks to Lemma \ref{sigma_density}:
\[
\begin{split}
E(t\wedge \sigma_0)&=\int_0^t\lbar(s)ds= \int_0^t \frac{ds}{(1+2s)^{3/2}}=1-\frac{1}{\sqrt{1+2t}}.
\end{split}
\]
Thus $E(Z_{t\wedge \sigma_0})=1/\sqrt{1+2t}$. Since $Z_{t\wedge \sigma_0} \equiv 0$ when $t\ge \sigma_0$, and $P(\sigma_0 > t)= (1+2t)^{-3/2}$ we get
\[
E\left(  Z_t \mid \sigma_0 > t \right)= \frac{E(Z_{t\wedge \sigma_0})}{P(\sigma_0 > t)}= 1 + 2t.
\]
This completes the proof.
\end{proof}

It is a well-known result due to Yaglom that for the critical GW process, without emigration, the law of the conditional random variable $Z/n$ given $\sigma_0 > nt$ converges to an Exponential random variable as $t$ is kept fixed and $n$ tends to infinity. Hence, it is natural to expect that a similar limit theorem holds for GW($-1$). However the known proofs in the case of no emigration does not carry over to the present case. We present a different proof using martingales. 

\begin{lemma}
Consider the (generalized) Laguerre orthogonal polynomials
\[
L_n^{(\alpha)}(x) = \frac{x^{-\alpha}e^{x}}{n!} \frac{d^n}{d x^n} \left( e^{-x} x^{n+\alpha}  \right), \qquad x\ge 0, \; n=0,1,2\ldots.
\]
Fix an $x \ge 0$ and define the function
\[
f(n) = L_{n-1}^{(3/2)}(x), \quad n \ge 1, \qquad f(0):=L^{(3/2)}_{-1}(x)\equiv 0,
\]
Then, if $Z$ is a GW($-1$) process then
\eq\label{laguerremartingale}
M^x_t= e^{2xt}f(Z_t), \quad t\ge 0,
\en
is a martingale. In particular, for $x=0$, the process
\eq\label{scalefunctionwhat}
M^0_t=\combi{Z_t+1/2}{ Z_t-1}, \quad t\ge 0,
\en
is a martingale which is the scale function of the Markov chain. 
\end{lemma}

\begin{proof}
We recall the following recursions satisfied by the Laguerre polynomials:
\begin{eqnarray}
L_n^{(\alpha)}(x) &=& L_n^{(\alpha+1)}(x) - L_{n-1}^{(\alpha+1)}(x)\\
nL_n^{(\alpha)}(x)&=& (n+\alpha) L_{n-1}^{(\alpha)}(x) - x L_{n-1}^{(\alpha+1)}(x).
\end{eqnarray}
Please see Abramowitz and Stegun \cite{AS} for a proof.

Now, for all $n \ge 1$, we have
\[
\begin{split}
\mathcal{L}f(n)&= 2n\left[ f(n+1) - f(n) \right] + (2n+1)\left[ f(n-1) - f(n) \right]\\
&=2n\left[ L_{n}^{(3/2)}(x) -  L_{n-1}^{(3/2)}(x) \right] + (2n+1)\left[ L_{n-2}^{(3/2)}(x) -  L_{n-1}^{(3/2)}(x) \right]\\
&= 2n L^{(1/2)}_n(x) - (2n+1)L^{(1/2)}_{n-1}(x).
\end{split}
\]
Here we have used the first of the two recursions. 

Now using the second recursion, we get
\[
\begin{split}
\mathcal{L}f(n)&=2\left[ nL_n^{(1/2)}(x) - (n+1/2) L_{n-1}^{(1/2)}(x)    \right]\\
&=-2xL^{(3/2)}_{n-1}(x) = -2xf(n).
\end{split}
\]
Thus
\[
2xf(n) + \mathcal{L}f(n) \equiv 0, \qquad n \ge 1,
\]
which proves that $M^x$ is a local martingale. The required integrability conditions for being a true martingale follows by comparison with the critical GW (no emigration) which is stochastically larger. 

The claim \eqref{scalefunctionwhat} follows from a standard formula of $L^{(\alpha)}_n(0)$. Note that $M_t^0$ gives the scale function of the Markov chain since it is increasing and vanishes at zero. Also note that when $Z_t$ is large, $M_t^0$ is proportional to $Z_t^{3/2}$ which is the scale function of the diffusion approximation stated below.  
\end{proof}

We can now prove our corresponding version of Yaglom's theorem. 

\begin{lemma}\label{Yaglom}
Let $Z$ be a GW($-1$) process. The sequence of laws of random variables
\[
\left\{ \left(\frac{Z_{nt}}{n} \mid \sigma_0 > nt\right), \; n \ge 1\right\}
\]
converges to an Exponential random variable with mean $2t$.
\end{lemma}

\begin{proof} Using Optional Sampling Theorem to the martingale $M_t^x$ in \eqref{laguerremartingale} for the bounded stopping time $t \wedge \sigma_0$ we get
\[
1 = E M^x_{t\wedge \sigma_0} = E e^{2x t} f(Z_t) 1\{ \sigma_0 > t  \}.
\]

Since we know that $P(\sigma_0 > t)= (1+2t)^{-3/2}$, we can rearrange the last equation to derive 
\eq\label{expeclaguerre}
E\left[ L_{Z_t-1}^{(3/2)}(x)  \mid \sigma_0 > t \right] = e^{-2xt}(1+2t)^{3/2}.
\en

Now an alternative series expansion of the Laguerre polynomials is the following
\[
e^{-x}L_n^{(\alpha)}(x) = \sum_{i=0}^\infty (-1)^i \combi{\alpha + n + i}{n} \frac{x^i}{i!}.
\]
Putting expectations on both sides as in \eqref{expeclaguerre} we get
\[
\begin{split}
e^{-(1+2t)x}(1+2t)^{3/2} &= E\left[ e^{-x} L_{Z_t-1}^{(3/2)}(x)  \mid \sigma_0 > t \right] \\
&= \sum_{i=0}^\infty (-1)^i E\left[ \combi{1/2 + Z_t + i}{ Z_t-1} \mid \sigma_0 > t \right] \frac{x^i}{i!}.
\end{split}
\] 

Both sides are entire expressions in $x$, and hence by comparing coefficients we get
\[
E\left[ \combi{Z_t + i + 1/2}{ Z_t-1} \mid \sigma_0 > t \right] = (1+2t)^{i+3/2},\quad i=0,1,2,\ldots.
\]
Replacing $t$ by $nt$ we get
\eq\label{combibound}
E\left[ \combi{Z_{nt} + i + 1/2}{ Z_{nt}-1} \mid \sigma_0 > nt \right] = (1+2nt)^{i+3/2}.
\en
For the rest of the proof we will denote expectations with respect to the conditional law of $Z_t$ by $E^*$.

By using Stirling's approximation formula 
\[
\Gamma(az+b) \sim \sqrt{2\pi} e^{-az} (az)^{az+b-1/2},
\]
we readily obtain that for fixed $y > 0$, as $n$ tends to infinity,
\[
\combi{ny + i + 1/2}{ ny-1}\sim \frac{(ny)^{i+3/2}}{\Gamma(i+5/2)}.
\]
Hence
\[
\lim_{n\rightarrow \infty} E^*\left( \frac{Z_{nt}}{2nt} \right)^{i+3/2} = \Gamma(i+5/2).
\]

The algebra generated by the class of functions $\{1\} \cup \{ y^{3/2}\cdot y^{i}, \; i\ge 0\}$ can uniformly approximate any continuous function on compact sets. Hence an argument similar to the \textit{moment method} identifies the limiting distribution of $Z_{nt}/2nt$ uniquely as Exponential with mean one. 
This completes the proof.
\end{proof}

Finally we state a standard diffusion approximation lemma. 

\begin{lemma}\label{diffuseapprox} Let $Z$ be a GW(-1) starting with $nx$ individuals. Then, as $n$ tends to infinity, the rescaled process
\[
\left( \frac{1}{n} Z_{nt}, \quad t \ge 0   \right)
\]
converges in law to a BESQ process with drift $-1$ starting from $x$. 
\end{lemma}

For a proof (without the drift term) see Chapter 9 on Ethier and Kurtz \cite{EK}. The proof with the drift is almost identical. 
The following collection of results can be found in the article by G\"oing-Jaeschke and Yor \cite{yornbesq}.

\begin{lemma}\label{lem:trev} Let $Q_x^{\theta}$ denote the law of a BESQ process of dimension $\theta\in \rr$ and starting from $x > 0$, until it hits zero. 
\medskip

For any $\theta > -2$ and any $x >0$, $Q_x^{-\theta}(T_0 < \infty )=1$, while, for $\theta \le 2$, one has $Q^\theta_x(T_0 < \infty)=0$.
\bigskip

Moreover, for $\theta \le 2$, we have
\begin{enumerate}
\item[(i)]  $T_0$ is distributed as $x/2G$, where $G$ is a Gamma random variable with parameter $(\theta/2 +1)$.
\item[(ii)] The transition probabilities $p_t^\theta(x,y)$ for $x,y >0$ satisfy the identity 
\[
p_t^{-\theta}(x,y)=p_t^{4+\theta}(y,x).
\] 
\end{enumerate}
\end{lemma}

As a final remark note that these transition probabilities are explicitly known as non-central chi-square distributions.

\begin{figure}[t]
\centering
\includegraphics[width=3.5in, height=3in]{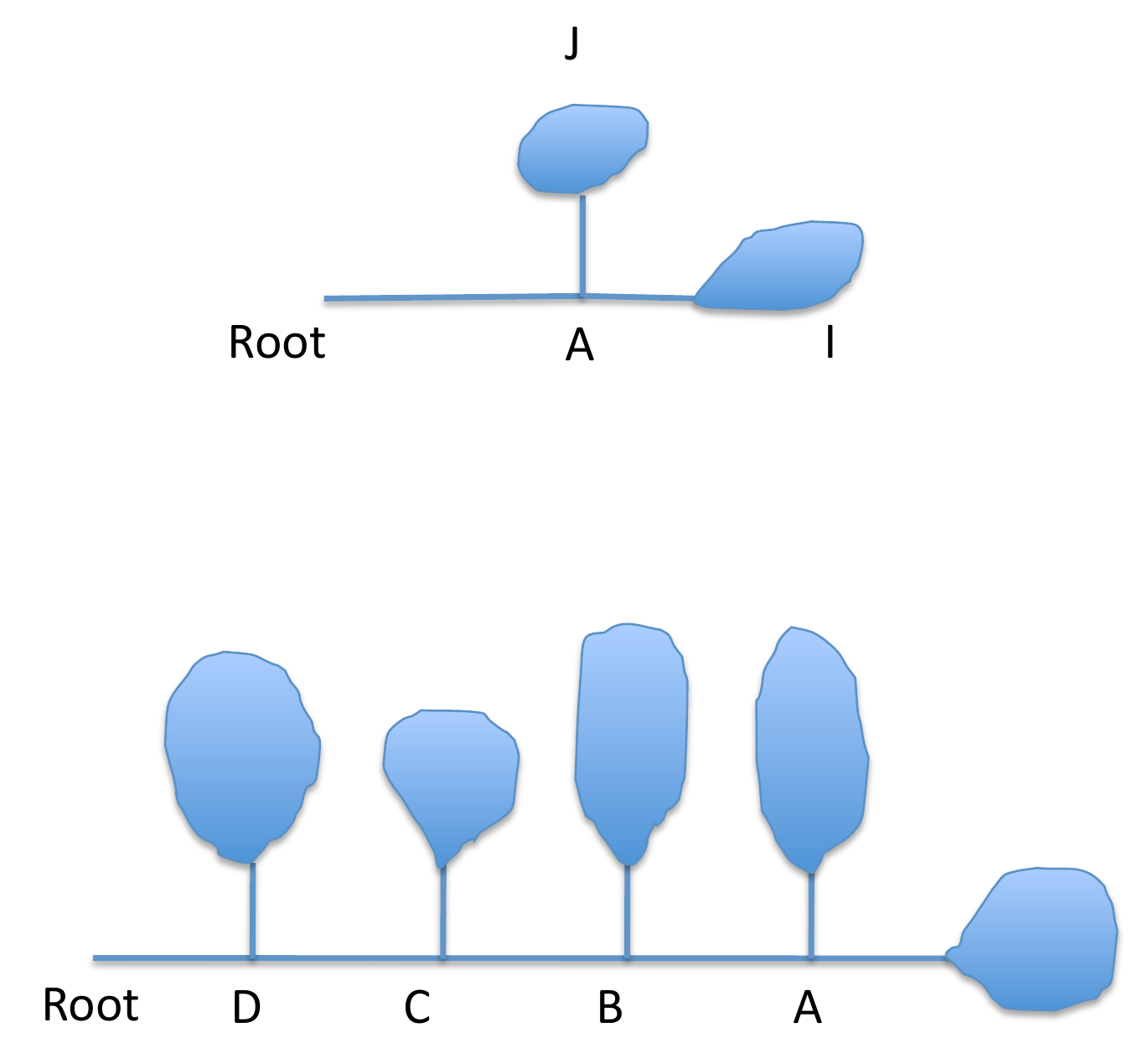}
\caption{Evolution of a street}
\label{excursionbegins}
\end{figure}

\newpage

\section{Evolution of a street}

Consider the Poissonized Markov chain as described in the Introduction. We consider a \textit{spine}, i.e., a branch, in the tree that supports several subtrees rooted on it. As the Poissonized chain runs its course, these subtrees either continue to exist or vanish at some point, and new subtrees appear supported on the spine. To distinguish the \textit{time} according to which the Markov chain runs, we will call this indexing set to be \textit{level}. Thus the Poissonized chain at level $a$ would mean the state of the chain at the end point of the time interval $[0,a]$.

Our purpose in this section is to describe, at any level $a$, the respective sizes of the sequence of existing subtrees supported on the spine (if any) as we move away from the root. Of course, this depends on the initial condition. We begin our discussion starting with a spine that supports just one root and two subtrees (see the top image in the Figure~\ref{excursionbegins}). We call this a sapling. The only internal vertex is marked by $A$. As long as this internal vertex exists, it corresponds to an internal edge to its left where, at rate one, new rooted subtrees can grow. At some point of time the subtrees can look like the second image in Figure~\ref{excursionbegins}. Rooted subtrees have grown at internal vertices $A, B, C$ and $D$. There are five internal vertices including the root with equal edge lengths between them. There are five subtrees growing from these internal vertices, one each from vertices $D$, $C$, $B$, and two from $A$. Call the rightmost subtree growing from vertex $A$ the clock subtree.

\begin{figure}[t]
\centering
\includegraphics[width=3.5in, height=3in]{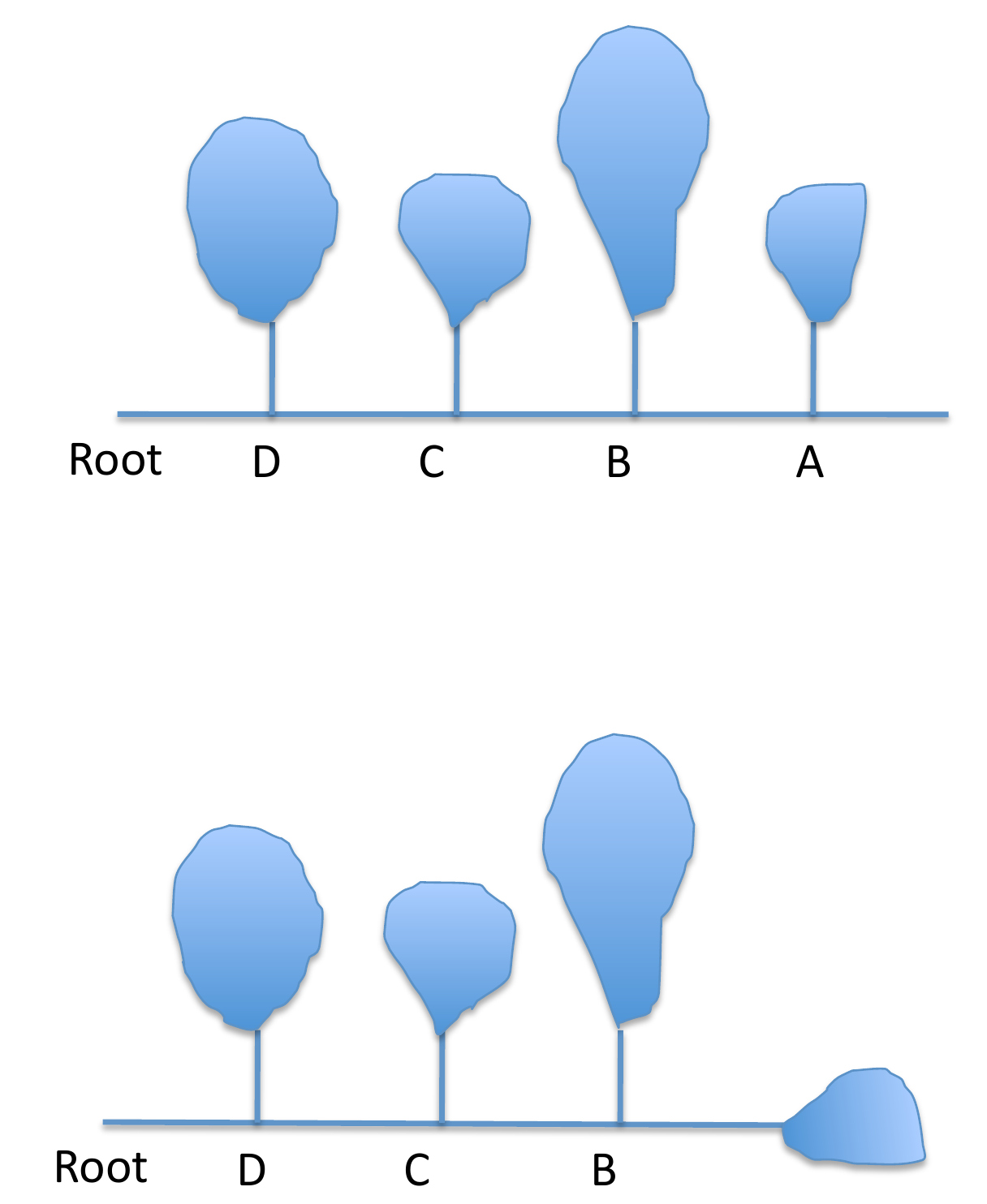}
\caption{Evolution of a street}
\label{fig_excursion1}
\end{figure}

Note that each of the internal vertices correspond to an internal edge on their left on which new leaves sprout at rate one. Hence one can associate a natural family order among these subtrees where we say subtree 1 is a child of subtree 2 if subtree 1 sprouted on the internal edge that was on the immediate left of the root of subtree 2. Clearly each existing subtree produces children at rate one as long as it is rooted at an internal vertex. As we have argued before, the number of leaves of these subtrees follows a GW($-1$) process, and hence the lifetime of each subtree is given by the distribution $L$ in Lemma \ref{sigma_density}.

Note that the internal vertex $A$ is special since two rooted subtrees have grown from $A$. In all these images the rightmost clock subtree has a special status. This subtree has no internal edge between itself and the next subtree on its left. Hence, this subtree \textit{does not} produce any child. Naturally the distribution of lifetime of this tree is also given by $L$. When this subtree dies out, the internal vertex to which it is attached dies too. Note that, this internal vertex may not be $A$, since the subtree that started with leaf $J$ (see previous figure) could have vanished earlier.  However, it is some internal vertex, and possibly the Root itself. 

Suppose, this is not the Root, but some other internal vertex. In that case the resulting picture is stochastically similar. In Figure~\ref{fig_excursion1} the top image represents the street the moment before the clock subtree attached on the right of $A$ dies. According to the rule of the Markov chain, the internal vertex $A$ disappears, and after some edge rescaling, the vertex $B$ becomes the final internal vertex which gives rise to two subtrees out of which the rightmost one is the new clock subtree that does not produce any children.   

This process continues until no internal vertex is left, save the root, and ultimately every leaf dies. 

Streets, as defined in Definition \ref{whatisastreet}, arise naturally in this context. Consider the sapling in Figure \ref{excursionbegins}. As we will see later, we will attach positive numbers $\alpha_I$ and $\alpha_J$ called the ages of the subtrees $I$ and $J$.  Then a street $(\nu, I, R)$ is defined as follows. There is only one internal vertex other than the root. Hence we take $I=1$, corresponding to the fact that this internal vertex is at a unit distance from the root. The measure $\nu$ has only one atom, at $1$, and $\nu(\{ 1\})=\alpha_J$. The clock $R$ is also taken to be $\alpha_I$.

As the Poissonized chain evolves, so does the street. Suppose at level $a$ we arrive at the configuration given in the second image in Figure \ref{excursionbegins}. Then, at level $a$, the street is given by the following. There are four internal vertices including the root with equal edge lengths between them. There are five subtrees growing from these internal vertices, one each from vertices $D$, $C$, $B$, and two from vertex $A$. For each of these subtrees, other than the clock subtree, we define the their ages as follows. If the root of the subtree did not exist at time zero (i.e., the internal vertex did not exist), the the age of the subtree is given by the difference between levels $a$ and the level at which that internal vertex was born. If the internal vertex existed at time zero, then its age is exactly $a$. For the clock subtree, the age is the sum of the amount of time it is the clock subtree plus its age just when it became the clock subtree (i..e, the then current clock subtree vanished on its right, or that it was the clock subtree to begin with).

Let $\alpha_1, \ldots, \alpha_5$ denote the ages of the subtrees growing arranged in order away from the root. Then, to define the street, we take $I=4$, $R=\alpha_5$, and $\nu$ is given by
\[
\nu\left( i \right) = \alpha_i, \qquad i=1,2,3,4.
\]

It is now intuitive that the evolution of the street starting from a sapling plays the role of an excursion of the evolution of the street under the general Poissonized Markov chain. Our aim is to describe the limiting excursion measure of this process. 

We remark at this point that this definition of the street will be slightly altered in the following text when we will scale edge lengths. This will be reflected in the scaling of the $(\nu, I, R)$. In particular, note that measure $\nu$ keeps track of both the distances on the spine and some information about each individual subtrees.

\begin{figure}[t]
\centering
\includegraphics[width=5in, height=3in]{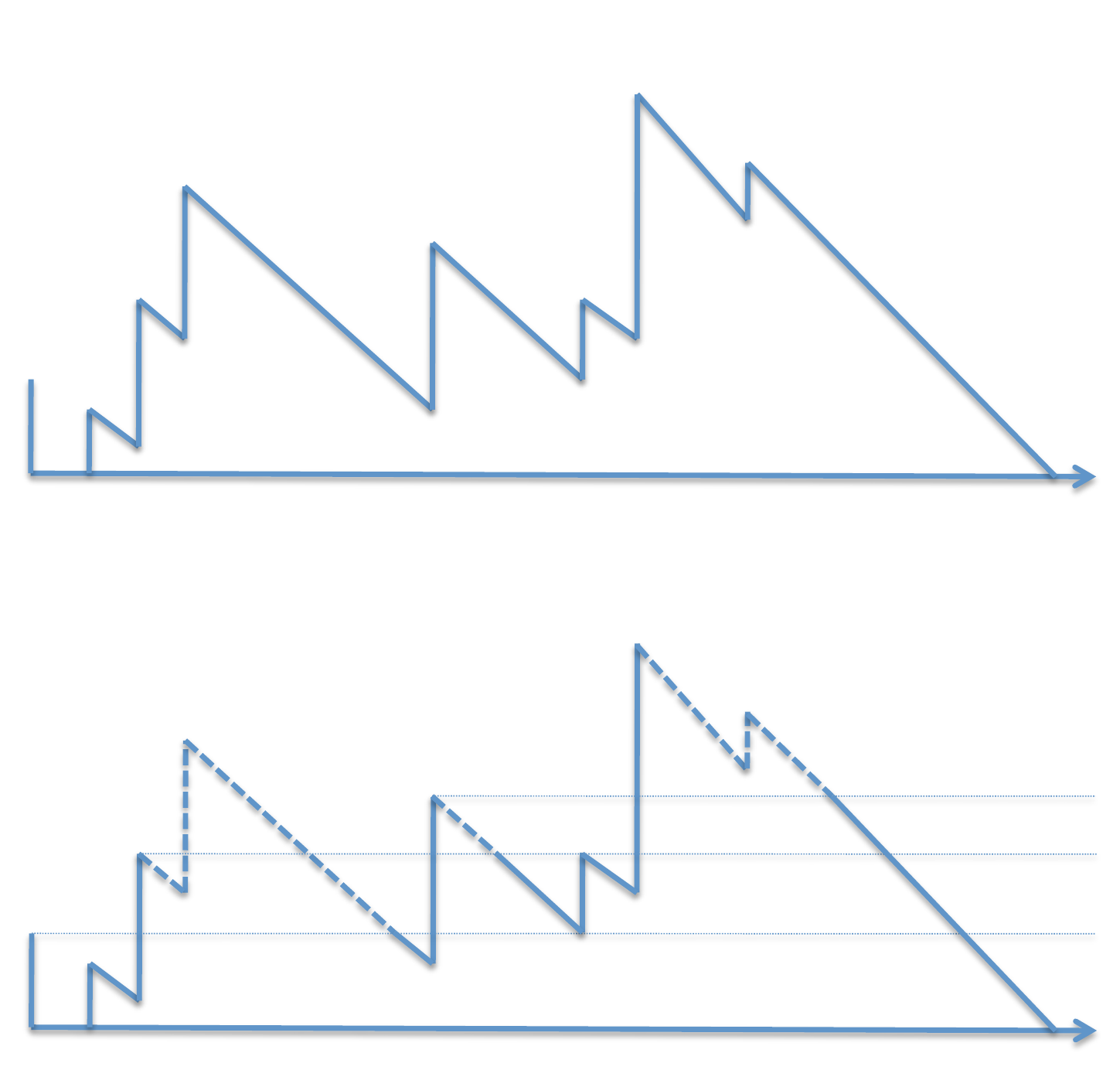}\\

\includegraphics[width=5in, height=3.5in]{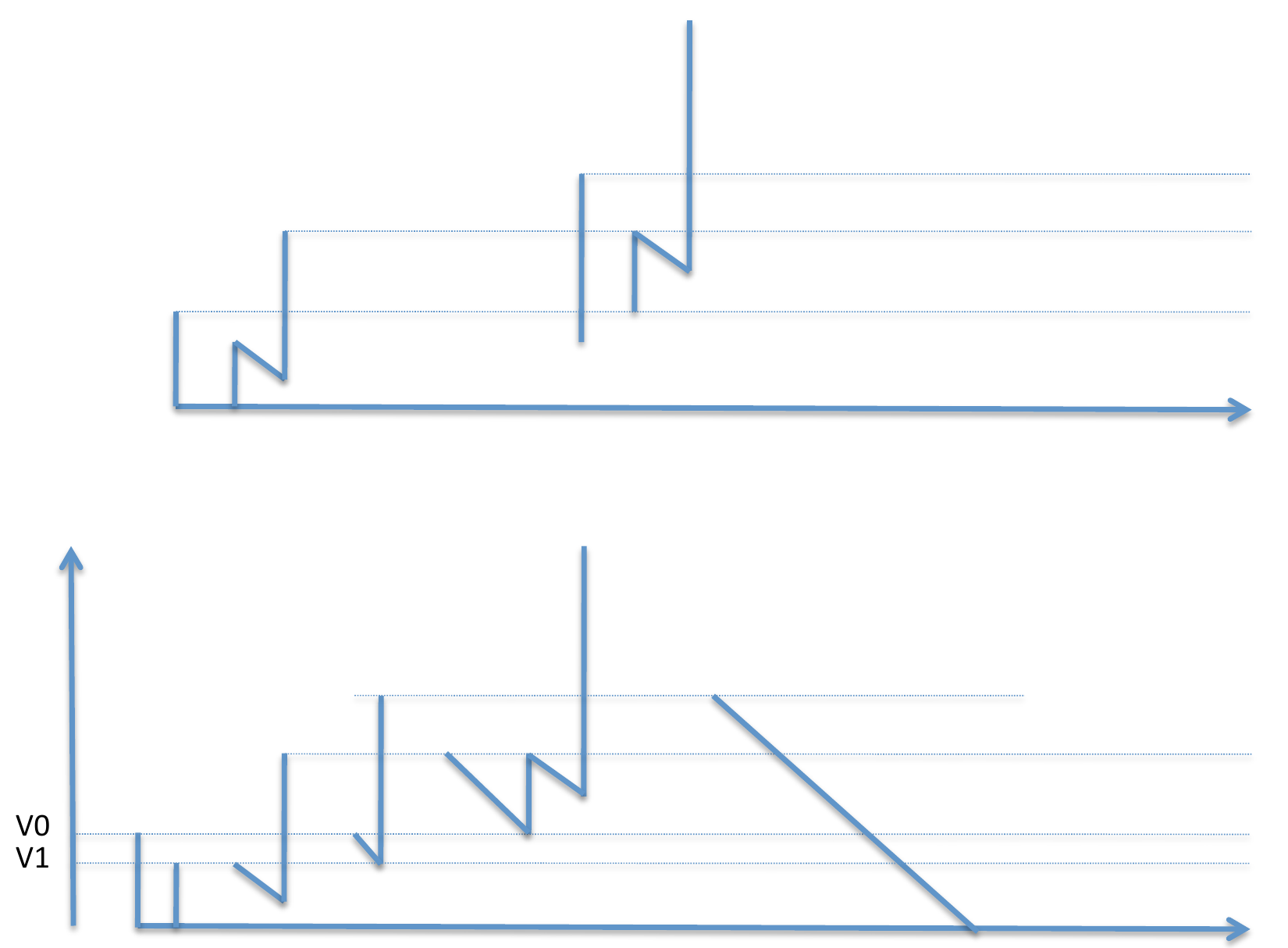}
\caption{The reduced JCCP}
\label{redJCCP1}
\end{figure}
\bigskip
\bigskip

Let us forget for the moment that the rightmost subtree does not produce children. In fact, consider that the subtree $I$ in Figure~\ref{excursionbegins} never dies. In that case, every subtree that grows survives until a random time whose distribution is $L$, and during that lifetime produces children at rate one on the left. The resulting process is nothing but a splitting tree, starting with one individual, and with a lifespan measure given by $L$, except that the children are produced to the left. This, however, can be easily corrected by reversing our direction of viewing and looking at the subtrees from the end away from the Root.  A typical  JCCP  is the top image of Figure~\ref{redJCCP1}.

Let us now explain the sequence of images in Figure \ref{redJCCP1}. The top most image represents the JCCP of the splitting tree if  we remove the constraint that when the rightmost subtree on the street dies whichever is then the rightmost does not produce any children for the rest of its lifetime. Since we are looking at the tree from its rightmost end, the leftmost jump in the top image (the singleton jump that is not connected with the rest of the image) of the JCCP represents the lifetime of the subtree that started with the subtree $I$, the first clock subtree. The rest of the image is the JCCP of a splitting tree with one ancestor and a lifetime distribution given by $L$. Since the mean of $L$ is one, the resulting splitting tree is critical and hence the JCCP hits zero almost surely. 

The second image in the sequence has two types of lines, solid and dotted. The dotted lines represents the lineages in the tree that do not exist since the rightmost clock subtrees on the street do not produce any children. The horizontal dotted lines represent the death times of the current clock subtree. Note that a subtree could have produced children before it becomes the clock subtree that does not produce any child.

The third image represents represents the JCCP when the dotted parts are removed. This gives us, at every level, the age process of the current survivors and their distances from the root, and the age of the current clock. Suppose that the initial two  jumps are given by $(V_0, V_1)$, then at every level we obtain a street. 

The fourth image displays that the third image corresponds to a sequence of \textit{excursions} of the underlying L\'evy process that defines a JCCP starting from a level and ending when it crosses another level. More precisely, define the running maximum process in a usual manner. The $n$th excursion starts at the end-point of the $(n-2)$th excursion and ends when it jumps across the level of the current maximum. This continues until an excursion hits zero and we stop there. To make the distinction of the usual JCCP and the new one, we will call it the reduced JCCP. Without being pedantic we write down the following definition.

\begin{defn}\label{reducedJCCP}
For a finite splitting tree, a reduced JCCP is a path obtained from a full JCCP corresponding to the evolution of a sapling (Figure \ref{excursionbegins}) following the rules described in Figure ~\ref{redJCCP1}. The initial pair of jumps that forms the sapling $(V_0, V_1)$ will be called the initial parameters of of the reduced JCCP.
\end{defn}

Our underlying intuition is the following. The sequence of jumps above the current maximum for the full (not reduced) L\'evy process which runs at all time is the ladder-height process of another L\'evy process which is a Subordinator. We will show that this process converges to the Stable$(1/2)$ subordinator which will give us the ages of the clock subtrees. Moreover, given the start and the end levels, each excursion is independent of the past and produces a limiting Poisson point process of ages across various levels. The final age process is a superposition of the point processes from various excursions which gives us the measure on the street.

As the Poissonized Markov chain proceeds, the distribution of leaves in each subtree can be described by describing the evolution of the triplet that consist the street. The scaling constant will be kept fixed throughout the evolution of the chain. Finally, note that the measure on the street can be empty (i.e., $I=0$).

Now, consider the evolution of streets as displayed in Figure 4. The last result shows the limiting age process of the family as it evolves under the Poissonized Markov chain. However, for the final picture, because of the way we have drawn streets in Figure 4 (Evolution of a street), we need to time-reverse our point processes at every level $a$ to get place the root at the left-most end and such that the rightmost subtree does not reproduce.

\subsection{Infinite splitting forest with lifespan measure $L$}

We start with the following set-up. Consider the probability distribution of the time to extinction of  a Galton-Watson tree with an emigration rate one. Explicitly, as derived in the previous section,
\[
\lbar(u)= 1- L(u) = (1+2u)^{-3/2} 
\]
Let $N_t$ denote a Poisson process with rate one and let $\zeta_1, \zeta_2, \ldots$ be iid distributed according to $L$. Then, define
\eq\label{cpp}
X_t = -t + \sum_{i=1}^{N_t} \zeta_i, \qquad t \ge 0,\quad X_0=0
\en
Thus $X$ is a spectrally positive L\'evy process whose Laplace exponent is given by Lemma \ref{emi_compu}:
\eq\label{whatisvarphi}
\varphi(\theta) = \theta - 1 + \int_0^\infty e^{-\theta r} L(dr) = \frac{1}{\sqrt{2}}\theta^{3/2} e^{\theta/2} \Gamma\left(1/2,\theta/2 \right).
\en  

If we consider the process $X$ whenever it is nonnegative and including the jumps crossing zero, the process is the JCCP of a countable forest of binary trees which are all attached at the root to the half line. At time zero, a countable number of individuals already exist with their respective ages. The Poissonized Markov chain is allowed to run on this forest where new trees can grow on the half-line in between two existing trees. For one individual as an ancestor we would stop this process once we hit zero. However, for this forest we would consider the process running for all time. 

At any point in time, the ages of the trees, when indexed by the order they appear, away from the origin can be thought of as a point process on the half-line. As time moves, this becomes a Markov process on the space of point processes. In this section we identify the limiting Markov process as a PAPP. Note that, since $L$ is a probability distribution with mean one, the resulting forest is critical. 

The main subtlety in doing this is that, by Lemma \ref{sigma_density}, there is no proper scaling under which the contour process $X$ converges to a spectrally positive L\'evy process. However, as we will show, the age-process does converge. This will be achieved by a combination of classical Wiener-Hopf techniques and results on ruin probabilities of spectrally positive L\'evy processes. See Bertoin \cite[chapter VII]{B}.

We start with a lemma which follows since $X$ is a martingale. 

\begin{lemma}
The process $X_t$ is point recurrent.  
\end{lemma}

\begin{defn} Let us make the following definitions which we will follow in the rest of the article. 
\begin{enumerate}
\item[(i)] For any spectrally positive L\'evy process $X$, we will denote the dual process $-X$ by $\dual X$, which is now a spectrally negative L\'evy process. 
\item[(ii)] We consider a Poisson point process as in \eqref{cpp} with a continuous jump distribution $L$. Consider a nonnegative level $a$. Define stopping times
\[
\varrho_a = \inf \left\{ t > 0:\; X_t = a    \right\}, \qquad \tau_a=\inf\left\{t> 0:\; X_t > a  \right\}.
\] 
Each $\tau_a$ will be called an upcrossing of level $a$, while each $\varrho_a$ will be called a downcrossing or a return to level $a$. Since $X$ makes only upward jumps, it is clear that, almost surely, any two occurrences of $\varrho_a$ contains a unique occurrence of $\tau_a$, and vice-versa. 
\item[(iii)] We will call the random variable $X_{\tau_a} -a$ to be the overshoot at level $a$ and the random variable $a - X_{\tau_a-}$ to be the undershoot at level $a$.
\item[(iv)] When $X_0=a$, the law of the overshoot $I=X_{\tau_a}-a$ is called the ascending ladder height distribution and is clearly independent of $a$. 
\item[(v)] For the above process $X$, we denote by $X^a$ the process obtained by time-reversal at $\varrho_a$, i.e., $X^a_t := a-X_{(\varrho_a-t)^+}, \quad t \ge 0$. The following equality is easy to see for arbitrary random walks: the law of the process $X^a$ in the time interval $[0, \varrho_0(X^a)]$ is the same as the law of $X-a$ in the interval $[0, \varrho_a]$ when $X_0=a$. In particular, the random variable $J= a - X_{\tau_a-}$ also has the ascending ladder height distribution. 
\end{enumerate}
\end{defn}

We now define the scale function of the process $X$.

\begin{lemma}\label{scalefunction}
Let $W:[0,\infty) \rightarrow [0,\infty)$ be the unique continuous increasing function with Laplace transform
\eq\label{Wlaplace}
\int_0^\infty e^{-\theta x}W(x)dx = \frac{1}{\varphi(\theta)}=\frac{\sqrt{2}\theta^{-3/2}}{ e^{\theta/2} \Gamma\left(1/2,\theta/2 \right)}.
\en
Then for every $x,y \ge 0$, if we define $\sigma_{x,y}$ to be the stopping time $\sigma_{x,y}=\inf\{t > 0:\; X_t \notin [-y,x] \}$, then   
\[
\begin{split}
P^0\left( X_{\sigma_{x,y}} = -y\right)=\frac{W(x)}{W(x+y)}.
\end{split}
\]

Moreover, $W$ satisfies the following properties.
\begin{enumerate}
\item[(i)] $W(0)=1$ and, asymptotically as $x$ tends to infinity, we get $W(x)\sim \frac{3}{\sqrt{2\pi}} x^{1/2}$.
\item[(ii)] $W \star \lbar = 1 + W$, where $\star$ denotes the convolution operator. 
\item[(iii)] $W$ admits a continuous density on $(0,\infty)$.
\end{enumerate}
\end{lemma}

\begin{proof}
The process $\dual X$ as defined in the above lemma has only downward jumps (i.e., spectrally negative). The ruin probabilities of such a process is given by a scale function. The first claim in the lemma is hence a corollary of Theorem 8 in \cite[p.~194]{B}. 
\medskip

\noindent Proof of (i): For the asymptotics consider the Laplace transform of the function $W$ as $\theta$ approaches zero. The function $\Gamma(1/2,\theta/2)$ converges at $0+$. Hence, it follows that
\[
\int_0^\infty e^{-\theta x} W(x) dx \sim \sqrt{\frac{2}{\pi}}\theta^{-3/2}, \quad \text{as}\quad \theta \rightarrow 0+.
\] 
In other words, the Laplace transform of $W$ is regularly varying at zero with index $-3/2$ and a slowly varying function given by a constant.

We now apply a well-known Tauberian theorem for monotone densities (note that $W$ is an increasing function). See, for example, \cite[p.~10]{B} to infer that 
\eq\label{Wlimit}
W(x) \sim \sqrt{\frac{2}{\pi}} \frac{3}{2} x^{1/2} \sim \frac{3}{\sqrt{2\pi}} x^{1/2}, \quad \text{as} \quad x\rightarrow \infty.
\en

Similarly, as $\theta$ approaches infinity, we get (see \cite[p.~263]{AS}):
\[
\Gamma(1/2, \theta/2)\sim \left( \theta/2  \right)^{-1/2} e^{-\theta/2} \left[   1 - \frac{1}{\theta} + O\left( \theta^{-2} \right) \right].
\]
Hence
\[
\int_0^\infty e^{-\theta x} W(x) dx \sim \theta^{-1}, \quad \text{as}\quad \theta \rightarrow \infty.
\] 
An application of the same Tauberian theorem and appealing to the continuity of $W$ at $0+$ completes the proof of the lemma.
\medskip

\noindent Proof of (ii): Note that by usual transformation rules of Laplace transforms
\[
\int_0^\infty e^{-\theta x} W(x) dx= \frac{1}{\varphi(\theta)}= \frac{\theta^{-1}}{1 - \int_0^\infty e^{-\theta x}\lbar(x)dx}.
\] 
By rearranging terms we get (ii).
\medskip

Claim (iii) follows from (ii) by a convolution series expansion. 
\end{proof}

\begin{lemma}\label{ladderheight}
Fix an initial level $X_0=a\ge 0$.  Then the random variables $J_a= a - X_{\tau_a-}$ and $I_a=X_{\tau_a}-a$ both have a density on $(0,\infty)$ given by $(1+2s)^{-3/2}$, i.e.,
\[
P^a\left( I_a > u \right)= P^a(J_a > u)=\gbar(u):= \frac{1}{\sqrt{1+2u}},\qquad u \ge 0. 
\]
More generally, suppose $X_0=0$, then the joint density of $(I_a, J_a)$ is given by
\eq\label{jointij}
P^0\left( I_a\in du, J_a \in dv   \right)= L'(u+v) \left( W(a) - W(a-a\wedge v) + 1\{v\ge a\}\right) dudv,
\en
where $L'$ is the density of the jump distribution $L$. In particular,
\eq\label{densityj}
\begin{split}
P^0&\left( J_a \in dv \right) = \left( W(a) - W(a-a\wedge v) + 1\{ v\ge a\} \right)\lbar(v)dv, \quad \text{and}\\
P^0&\left(  I_0 + J_0 > x \right) = x\lbar(x) + \gbar(x), \qquad x\ge 0.
\end{split}
\en
\end{lemma}

\begin{proof}[Proof of Lemma \ref{ladderheight}]  
Since the law of the jump are given by $L$, the density of the ascending ladder height is given by $\lbar$. This follows from the classical Wiener-Hopf factorization and can be found in, for example, Theorem 5.7 in \cite[237]{Asmussen}.

For the joint distribution we use the expression derived in Doney and Kyprianou in \cite{DK}. See Example 8 (Spectrally positive proesses) on page 9 (use $\hat q=0$ and $U=W$ and integrate over the third variable).
\[
\begin{split}
P^0\left( I_a\in du, J_a \in dv   \right)&=  \left( L'(u+v)\int_0^{a\wedge v} W(a-dy)\right)du dv\\
&=L'(u+v) \left( W(a) - W(a-a\wedge v) + 1\{v\ge a \} \right) dudv.
\end{split}
\] 
Here we have used the fact that $W$ has a mass of $1$ at zero and after that it has a continuous density. The expression is consistent with the marginals when $a=0$. The final claim follows by trivial integrations. 
\end{proof}

Now, consider the sequence of returns of the process $X$ to the level $a$. When $X$ is the JCCP of the splitting tree, several such returns are performed until the process $X$ hits zero and gets killed. By the Markov property, the law of the process in between each such return is iid. Hence the number of such returns is determined by the probability 
\[
q_a= P^a\left( \inf_{0\le t \le \varrho_a} X_t > 0 \right).
\]

\begin{lemma}\label{probcross}
The required probability discussed above is given by
\eq\label{whatisqa}
\begin{split}
q_a &= P^a\left( \inf_{0\le t \le \varrho_a} X_t > 0 \right)=1 - \frac{1}{W(a)}. 
\end{split}
\en

Moreover, 
\begin{enumerate}
\item[(i)] the following holds for any two levels $0 \le a_0 \le a_1$. For any $0< x < a_1 $ we have
\eq\label{conditionalage}
P^{a_0}\left( J_{a_1}\in dx,\; \inf_{0\le t \le \tau_{a_1}} X_t > 0  \right) = P^0\left( J_{a_1-a_0} \in dx \right) - \frac{W(a_1-a_0)}{W(a_1)} P^0\left( J_{a_1} \in dx\right).
\en
\item[(ii)] Hence, when $a_0=a_1=a$, the conditional density of $J$ is given by
\eq
\begin{split}
r_a(v)&:=\frac{1}{dv}P^{a}\left( J_{a}\in dv\mid \inf_{0\le t \le \tau_{a}} X_t > 0  \right) =\frac{1}{dv}P^a\left( J\in dv\mid  \inf_{0\le t \le \varrho_a} X_t > 0   \right)\\
%\frac{q^{-1}_a}{W(a)} \frac{W(a-x)}{(1+2x)^{3/2}}, \quad 0< x< a.
&=q_a^{-1}\left( 1 - \frac{W(a-v)}{W(a)} \right) \lbar(v), \quad 0 < v < a.
\end{split}
\en
\item[(iii)] As $n$ tends to infinity 
\eq\label{whatisgstar}
g^*_a(v):=\lim_{n\rightarrow \infty} \sqrt{2n^3}\; r_{na}(nv) = \left(  1 - \sqrt{\frac{a-v}{a}} \right) v^{-3/2}, \qquad 0< v < a,
\en
exists as a rate function on the interval $(0,a)$.
\end{enumerate}
\end{lemma}

\begin{proof} Clearly $q_a$ is the probability that, starting at $a$, $X$ exits the interval $[0,a]$ through $a$. By a translation  and using Lemma \ref{scalefunction}, we get
\[
q_a = 1 - \frac{W(0)}{W(a)}= 1 - \frac{1}{W(a)}.
\]

For claim (i), we note the following.
\[
\begin{split}
P^{a_0}&\left(  J_{a_1}\in dx,\; \inf_{0\le t \le \tau_{a_1}} X_t > 0  \right)= P^{a_0}\left(J_{a_1}\in dx,\; \varrho_0 > \tau_{a_1}\right)\\
&= P^{a_0}\left( J_{a_1} \in dx \right) - P^{a_0}\left(J_{a_1}\in dx,\; \varrho_0 < \tau_{a_1}\right)\\
&= P^{a_0}\left( J_{a_1} \in dx \right) - P^{a_0}\left( \varrho_0 < \tau_{a_1}\right)P^0\left( J_{a_1}\in dx \right)\\
&= P^{0}\left( J_{a_1-a_0} \in dx \right) - \frac{W(a_1-a_0)}{W(a_1)}P^0\left( J_{a_1}\in dx \right).
\end{split}
\]

For claim (ii) we take $a_0=a_1=a > 0$. Note that the density is supported on $0\le v \le a$. Thus, from the above expression and using Lemma \ref{ladderheight}, we get the density to be
\[
\begin{split}
r_a(v) &= q_a^{-1}\left[  \lbar(v) - \frac{1}{W(a)} \left( W(a) - W(a- v)   \right) \lbar(v)\right]\\
&= q_a^{-1}\left(  \frac{W(a-v)}{W(a)} \right) \lbar(v).
\end{split}
\]

As a sanity check, note that from Lemma \ref{scalefunction}, claim (ii), we get
\[
\begin{split}
\int_0^a r_a(v)dv = q_a^{-1} \frac{1}{W(a)} W\star \lbar(a)= q_a^{-1} \left( 1 - \frac{1}{W(a)}  \right)=1,
\end{split}
\]
which shows that $r_a$ is indeed a probability density. 

For (iii) we simply use the asymptotics of the $W$ and $\lbar$ functions. 
\end{proof}

We now need some asymptotic estimates. 

\begin{lemma}\label{overshoot} Consider the Stable($1/2$) subordinator, i.e., the increasing L\'evy process with a jump distribution given by $\Pi(dy) = y^{-1/2}$, $y > 0$. Let $a,u > 0$. Then,
\begin{enumerate} 
\item[(i)] then the overshoot distribution, $I_{\tau_{na}}:=X_{\tau_{na}} - na$, satisfies
\[
\lim_{n\rightarrow\infty} P^{0}\left( I_{\tau_{na}} > nu \right)=\hbar_a(u),
\]
where $H_a(\cdot):=1-\hbar_a(u)$ is the distribution of the overshoot  above a level $a$ for the Stable($1/2$) subordinator. 
\item[(ii)] Moreover, the joint density $f_n$ of $(I_{na}/n, J_{na}/n)$ under $P^0$ admits the following pointwise limit on $(0,\infty)\times(0,\infty)$:
\[%\label{ijjoint}
\lim_{n\rightarrow \infty} f_n(u,v)= \frac{9}{2\sqrt{2\pi}} (u+v)^{-5/2} \left( \sqrt{a} - \sqrt{a- a\wedge v} \right).
\]
\item[(iii)] As $n$ tends to infinity, the marginal density $h_n$ of $J_{na}/n$ under $P^0$ converges as follows 
\[
\lim_{n\rightarrow \infty} h_n(v)= h_a(v):=\frac{3}{4\pi} v^{-3/2}\left( \sqrt{a} - \sqrt{a-a\wedge v}  \right),\qquad v\in (0,\infty).
\]
\item[(iv)] Consider two positive levels $0 < a_0 < a_1$. Consider the density of $J_{na_1}/n$, under the conditional probability $P^{na_0}\left(  \cdot \mid \inf_{0\le t \le \tau_{na_1}} X_t > 0 \right)$. Then, for $v\in (0,\infty)$, the sequence of densities converges pointwise to
\[
h^*_a(v)=\left( 1 - \frac{\sqrt{a_1 - a_0}}{\sqrt{a_1}} \right)\left[  h_{a_1 - a_0}(v)  - \frac{\sqrt{a_1- a_0}}{\sqrt{ a_1}} h_{a_1}(v) \right].
\] 
\end{enumerate}
\end{lemma}

\begin{proof} We first consider claim (i). Consider the process $X$ starting at zero and consider the supremum process $M_t= \sup_{0\le s \le t} X_s$. Then, plainly, $M$ is an increasing process that increases only when the process $X$ jumps above the level of the current supremum. Hence, its jumps are iid with a distribution given by the ascending ladder height distribution $G$ in Lemma \ref{ladderheight}. However, it is not Markovian, since the inter-arrival times of these jumps, although iid, are not Exponentially distributed. 

However, for any positive level $a$,  the overshoot distribution, $M_{\tau_a} -a$, does not depend on the distribution of the interarrival times. Thus the overshoot distribution is the same as the overshoot distribution above a for a discrete time random walk $S$ with a increment distribution $G$. Thus, let $\xi_1, \xi_2, \ldots$ be an iid sequence of random variables each with distribution $G$ and define $S_0=0$ and $S_n = \xi_1 + \ldots \xi_n$. Then $S_{\tau_a}$ has the same distribution as $M_{\tau_a}$ which has the same law as $X_{\tau_a}$. Thus the overshoot distribution is given by the law of $S_{\tau_a} - a$.

Since $\gbar(u)$ is asymptotically $(2u)^{-1/2}$, it follows that the rescaled process
\eq\label{convergencetostable}
Y^{(n)}_t=\frac{1}{n}S_{\lfloor t\sqrt{2n} \rfloor}, \qquad t \ge 0,
\en
converges in law (in $D[0,\infty)$) to a Stable($1/2$) subordinator $Y$. The result now follows by using standard arguments. 
\bigskip

Now for claim (ii) we use the explicit joint density of $(I_{na}, J_{na})$ from Lemma \ref{ladderheight}.
\[
f_n(u,v)= n^2 \frac{3}{2}(1+ n(u+v))^{-5/2} \left(  W(na) - W(n(a -  a\wedge v)) + 1\{ v\ge a\} \right).
\]
Now taking $n$ going to infinity and using the asymptotic properties of the function $W$ in Lemma \ref{scalefunction} we get
\[
\lim_{n\rightarrow \infty} f_n(u,v)= \frac{3}{2} (u+v)^{-5/2} \frac{3}{\sqrt{2\pi}} \left( \sqrt{a} - \sqrt{a- a\wedge v} + \lim_{n\rightarrow \infty} n^{-1/2} 1\{ v\ge a\} \right).
\]
This completes the derivation of (ii).
\bigskip

Claim (iii) is very similar and can be seen by integrating out $u$ from (ii). We use the explicit expression of $h_n$ from Lemma \ref{ladderheight}.
\[
\begin{split}
\lim_{n\rightarrow \infty} h_n(v)&= \lim_{n\rightarrow \infty} n (1+ 2nv)^{-3/2}\left(  W(na) - W(n(a- a\wedge v)) + 1\{v \ge a \} \right)\\
&= \frac{3}{\sqrt{2\pi}} (2v)^{-3/2}\left( \sqrt{a} - \sqrt{a-a\wedge v}  \right)= \frac{3}{4\pi} v^{-3/2}\left( \sqrt{a} - \sqrt{a-a\wedge v}  \right).
\end{split}
\]

\bigskip
Claim (iv) is a direct consequence of claim (iii) via Lemma \ref{probcross} claim (i). This is because
\[
\begin{split}
P^{na_0}&\left( \frac{J_{na_1}}{n} \in dv \mid \tau_{na_1} < \varrho_0 \right)= \frac{1}{P( \tau_{na_1} < \varrho_0)}P^{na_0}\left( J_{na_1}/n\in dv,\; \inf_{0\le t \le \tau_{na_1}} X_t > 0  \right)\\
&= \left( 1 - \frac{W(n(a_1 - a_0))}{W(na_1)} \right)\left[  P^0\left(  J_{n(a_1 - a_0)} \in dv\right)  - \frac{W(n(a_1- a_0))}{W(n a_1)} P^0\left(  J_{na_1 } \in dv\right) \right].
\end{split}
\]
Taking limit as $n$ goes to infinity and using (iii) above we get the limit as
\[
\left( 1 - \frac{\sqrt{a_1 - a_0}}{\sqrt{a_1}} \right)\left[  h_{a_1 - a_0}(v)  - \frac{\sqrt{a_1- a_0}}{\sqrt{ a_1}} h_{a_1}(v) \right].
\]
This completes the proof of the lemma.
\end{proof}

\subsection{The age process} Consider now the process $\{ X_t,\; 0\le t < \infty \}$, where $X_0=0$. We are going to construct a family of point processes, called the age process, which will be indexed by levels $a \ge 0$.

\begin{defn}\label{ageprocesslevy} (\textbf{The age process for the infinite forest}.)
Consider the sequence of successive upcrossings of a level $a\ge 0$, i.e., $\tau_a(1), \tau_a(2), \ldots$,
and their corresponding  values of the undershoots $J_a(i)= a-X_{\tau_a(i)-}$.
Then the point process defined by $\left\{ \left( k, J_a(k) \right),\; k=1,2,\ldots   \right\}$
seen as a family of point processes indexed by the level $a$ is called the age process. 

The rescaled age process will be the family of point processes 
\eq\label{whatisageprocess}
\age_n(a):=\sum_{k=1}^\infty \delta_{(k/\sqrt{2n}, J_{na}(k)/ n)}
\en
scaled by a factor $n$.
\end{defn}

For every level $a$, the sequence $(J_a(k), \; k \in \mathbb{N})$ is a sequence of iid random variables with distribution function $G$ via Lemma \ref{ladderheight}. Since, for any $u > 0$, we get
\[
\lim_{n\rightarrow \infty} \sqrt{2n} \gbar\left( nu \right) = u^{-1/2},
\]
the sequence of measures defined by 
\[
\mu_n\left([x,\infty)\right) = \gbar\left( [nx,\infty) \right)
\]
satisfies that $\sqrt{2n}\mu_n$ converges vaguely to the $\sigma$-finite measure $\mu[x,\infty)= x^{-1/2}$.

By Lemma \ref{PPPconvergence} the rescaled age process converges weakly, as $n$ tends to infinity, to a PPP on $\rr^+\times \rr^+$ with an intensity measure $\leb \times \mu$.
We sum this up in the following lemma.

\begin{lemma}\label{onedimconv}
Fix a level $a \ge 0$. Consider the sequence of scaled age-process at level $a$, i.e., $\{ \age_n(a),\; n\in \mathbb{N}\}$. Then, as $n$ tends to infinity, the above sequence converges in law to a Stable($1/2$)- PPP, $\age(a)$, on $[0,\infty)\times[0,\infty)$ with rate $\leb\times \mu$, where $\mu[b,\infty)=b^{-1/2}$. 
\end{lemma}

Thus the process $\{ \age_n(a), \; a \ge 0\}$ has marginal convergence to the Stable$(1/2)$-PPP. Now we ask the question of joint convergence. One can prove (as shown below) that the finite-dimensional distributions of this Markov process, after suitable rescaling, converges to what turns out to be the finite-dimensional distributions of a Stable$(1/2)$-PAPP. Kolmogorov's consistency theorem then establishes the existence of a limiting stochastic process of point process whose every marginal distribution is a Stable$(1/2)$ point process.

\begin{thm}
Consider $\{ \age_n(a),  \; a\in \rr \}$ as a stochastic process of point processes indexed by $a$. Then there is a limiting Stable $(1/2)$-PAPP, say $\{\age(a),\; a\in \rr\}$, such that the finite dimensional distributions of  $\age_n$ converges to the finite dimensional distributions of $\age$. Moreover, $\{\age(a), \; a\ge 0\}$ satisfies the restriction consistency property in Definition \ref{resconsis}.
\end{thm}

\begin{proof}
Let us start with convergence of a pair $(\age_n(a_0), \age_n(a_1))$ for some pair of numbers $a_0 < a_1$. By translation, we can assume that $a_0=0$ and $a_1=a > 0$. 

Note that, every upcrossing above level zero has a chance of producing a finite number of upcrossings above level $na$ before the process returns to zero. Let's denote by $J$ the undershoot at level zero (i.e., $-X_{\tau_0-}$) and $\Pi$ to be the finite point process (possibly empty) of undershoots at level $a$ associated with one excursion above zero. By the Markov property of $X$, for each excursion above zero, we get iid copies of $(J,\Pi)$ where the $\Pi$'s get concatenated to give us the age process at level $a$. By Lemma \ref{PPPconvergence}, one can expect the iid sequence of $(J,\Pi)$ labeled by their indices to converge to a PPP on the product space of $\rr^+$ and the space of marked point processes from which the conditional law of $\age(a)$ can be recovered by taking a projection.   

To do this precisely, we keep track of five random elements associated with one given excursion above level zero:
\begin{enumerate}
\item[(i)] $J_0$- The value of the undershoot at level zero.
\item[(ii)] $I_0$ - The value of the overshoot at level zero.
\item[(iii)] $\epsilon_a$ - The indicator which is one if the excursion above zero goes above $na$.
\item[(iv)] $J_{na}$ - The first undershoot at level $na$, in case $\epsilon_a=1$, otherwise $\emptyset$. 
\item[(v)] $Z_n$ - The rescaled age process of additional undershoots at level $na$ in case $\epsilon_a=1$, otherwise $\emptyset$.
\end{enumerate}

The joint distribution of all these five elements can be described as below. 
\begin{enumerate}
\item[(i)] The joint law of $(I_0,J_0)$ is given in Lemma \ref{ladderheight}. 
\item[(ii)] Given $(J_0,I_0)$, we now find the distribution of $\epsilon_a$. Note that $\epsilon_a$ is $1$ with probability $1$ if $I_0 \ge na$. If $I_0=nb < na$, the probability that the process $X$ starting from $I_0$ will reach zero before $na$ is given by $W(na-nb)/W(na)$. Thus, compactly, we can write 
\[
\begin{split}
P\left( \epsilon_a =0 \mid J_0, I_0=nb  \right)&= 1 - P\left( \epsilon_a =1 \mid J_0, I_0=nb  \right)\\
&= \frac{W(na-nb)}{W(na)}, \quad \text{where}\; W(x)=0, \; x< 0.
\end{split}
\] 
\item[(iii)] Now suppose we are given $J_0$ and $I_0 < na$, and $\epsilon_a =1$. For $n$ large, the law of $J_{na}$ is given in the limit by the conditional law described in Lemma \ref{overshoot} (iv).
\item[(iv)] The rest of the age process at level $a$, given $\{\epsilon_a=1\}$, is given by a sequence of iid undershoots, distributed as $r_a$, indexed by $\{1, \ldots, N_a\}$ where $N_a$ is distributed as Geometric$(1- q_{na})$, and is independent of everything else.
\end{enumerate}

Consider the joint distribution $\mu$ of $(J_0, I_0, \epsilon_a, J_a, Z)$ as a probability measure on 
$$
\Omega :=\rr^+\times \rr^+ \times \{0,1\} \times \rr^+\times \mx_{\rr^+\times \rr^+}.
$$ 
Now let $\mu_n$ be the joint distribution of $(J_0/n, I_0/n, \epsilon_{na}, J_{na}/n, Z)$ (note that $Z$ is already rescaled). That their joint density converges to the density of a $\sigma$-finite measure $\nu$  on $\Omega$ can be established by multiplying their marginal and conditional densities as given by various parts of Lemma \ref{overshoot}. Then, by Lemma \ref{densityvague} it follows that $\sqrt{2n}\mu_n$ converges vaguely $\nu$ (the verification that the densities are uniformly locally bounded follows easily from the explicit expressions). The description of $\nu$ follows from each of the limiting distributions, in particular, $Z$ converges to (either the empty measure, or) a PPP of an independent Exponential length.  

The existence of a Poisson point process on $\rr^+ \times \Omega$ with a rate function $\leb \times \nu$ follows from Lemma \ref{PPPconvergence}. The age process at levels $0$ and $a$ can now be recovered from this limiting point process by taking suitable projections.

\bigskip

A repetition of the same argument produces the limiting age process at any finitely levels $a_0 < a_1 < \ldots < a_k$. The consistency of the finite-dimensional distributions before taking the limit is preserved under vague convergence. By Kolmogorov's consistency theorem we have established the existence of the limiting stochastic process $\{\age(a), \; a\in \rr\}$.

Finally, the restriction consistency property is a finite-dimensional property that is obvious for $\age_n$ and is clearly preserved under the limit. 
\end{proof}

We will need to enlarge the scope of the last theorem slightly. For any $x,y >0$ consider the L\'evy process $X$, starting from $0$, until the exit time $\sigma_{xy}$ of the interval $(-ny, nx)$. Then the limiting age process can be similarly defined for this stopped process. 

\begin{thm}\label{stoppedage}
For any $0 < y < x$, let $X^{\sigma_{xy}}$ denote the stopped process $X$, which starts at $ny$, and stopped once it exits the interval $(0,nx)$. Construct the process $\age_n$ as in Definition \ref{ageprocesslevy}. Also define 
\[
\mathcal{I}^{(n)}_{x}=\frac{1}{{n}} \left( X_{\sigma_{xy}} - nx\right)^{+}.
\]

Then there is a limiting stochastic process of Point process $H^{x,y}$ and a limiting ransom variable $\mathcal{I}_x$ such that the joint distribution of $(\age_n, \mathcal{I}^{(n)}_x)$, under this law, converges in law to $(H^{x,y}, \mathcal{I}_x)$ as $n$ tends to infinity. Moreover the conditional distribution of $\mathcal{I}_x$, given $\mathcal{I}_x > 0$, admits the following description. Let $S$ denote a Stable$(1/2)$ subordinator, starting from $y$, and let $1-\hbar_a$ denote the distribution function of the overshoot above level $a$ as defined in Lemma \ref{overshoot}. Then, for all $b > 0$, we have
\[
P^{ny}\left(  \mathcal{I}_x > b \mid \mathcal{I}_x > 0 \right)= \left( 1 - \sqrt{\frac{x-y}{x}}  \right)\left[   \hbar_{x-y}(b) - \sqrt{\frac{x-y}{x}} \hbar_{x}(b)  \right].
\]
\end{thm}

\begin{proof}
We skip the details of the convergence of the age process which are exactly the same as in the previous case. 

The description of the conditional law of $\mathcal{I}_x$ follows by an argument similar to the proof of Lemmas \ref{probcross}.
\[
P^{ny}\left( I_{nx} \in du, \; \inf_{0\le t \le \tau_{nx}} > 0  \right)= P^0\left( I_{n(x-y)} \in du \right) - \frac{W(n(x-y))}{W(nx)}
P^0\left( I_{nx} \in du \right).
\]
Taking limits and using Lemma \ref{overshoot} (i), we get
\[
P\left(  \mathcal{I}_x > b \mid \mathcal{I}_x > 0  \right)=\left( 1 - \sqrt{\frac{x-y}{x}}  \right)\left[   \hbar_{x-y}(b) - \sqrt{\frac{x-y}{x}} \hbar_{x}(b)  \right].
\]
This proves the lemma.
\end{proof}

\subsection{Evolution of a street}

We now come to the main result of this section. Consider the reduced JCCP (Definition \ref{reducedJCCP}) in Figure \ref{redJCCP1}, in particular the final image. Note that the reduced JCCP has two initial parameters $(V_0, V_1)$ and a sequence of excursions. The two initial parameters determine the first two jumps denoted by $(V_0, V_1)$. These, in turn, determine the following: $V_1$ determines the level at which the first excursion starts and $\max(V_0, V_1)$, determines at what level we stop the excursion (unless it has already hit zero). Our claim is that, conditioned on the values of $(V_0, V_1)$ being suitably large, the process of streets of the reduced JCCP has a limit.

\begin{defn}
For any level $a\ge 0$ and a scaling parameter $n\in \mathbb{N}$, consider the reduced JCCP as in \eqref{whatisageprocess} with the initial parameters $nv_0, nv_1 >0$.  Consider the sequence of successive upcrossings of a level $na\ge 0$ in reverse order from right to left. We separate the undershoots below level $na$ in two parts, the left most jump (the rightmost in the reverse order), which corresponds to the clock subtree, and the rest of them ordered from right to left. 
Define the rescaled street at level $a$, denoted by $\street(a)$, to be the triplet
\[
\street_n(a) = \left(  \nu_n(a), I_n(a), R_n(a)   \right),
\]
where
\begin{enumerate} 
\item[(i)] $nR_n(a)$ is the age of the clock at level $na$ which is given by the length of the rightmost undershoot below level $a$ (if any).
\item[(ii)] $\sqrt{2n} I_n(a)$ is the number of upcrossings across level $na$, not counting the clock. In particular it can be zero.
\item[(iii)] Finally $\nu$ is the point process defined by
\[
\street_n(a):=\sum_{k=1}^{\sqrt{2n}I_n(a)} \delta_{(k/\sqrt{2n}, J_{na}(k)/ n)},
\]
where $J_{na}(\cdot)$ are the undershoots below level $na$ arranged from right to left. In case the jump exceeds $na$, the jump corresponds to rightmost of the two initial jumps, and we define $J_{na}(\cdot)$ to be the sum of $a$ and $\alpha_1$. 
\end{enumerate}
\end{defn}

\begin{thm}\label{thm:convstreets} 
Consider the sequence of processes of streets starting with $(V_0=nv_0, V_1=nv_1)$. Then there is a limiting family of random elements 
\[
\street_{v_0, v_1}(a)=\left\{(\nu(a), I(a), R(a)), \; a\ge 0\right \}
\]
such that, as $n$ tends to infinity, the finite-dimensional distributions (i.e., considering finitely many levels) of the rescaled street process, with initial parameters $(nv_0, nv_1)$, converges weakly to the finite dimensional distributions of the said limiting family.  

Moreover, the point process $\{\nu(a) \cap ( I(a)+1, R(a)), \; a\ge 0\}$ satisfies the restriction consistency property in Definition \ref{resconsis}.
\end{thm}
\medskip

\noindent\textit{Remark.} Adding the atom $(I(a)+1, R(a))$ is arbitrary. What we want to say is clear. If the clock subtree has survived through two levels $a_0 < a_1$, then one would find a jump across level $a_0$ such that it also crosses level $a_1$. 
\medskip

\begin{proof} This proof is now easy since there are no rate functions involved. Suppose that the initial parameters are $(nv_0, nv_1)$. Now consider a sequence of excursions, as in the reduced JCCP. Let $x_0=v_0$. The first excursion starts from level $nx_0$ and stops once it exits the interval $(0, nx_1)$, where $x_1\max(v_0, v_1)$. Conditioned on the first excursion not hitting zero, define $x_2=x_1 + \mathcal{I}^{(n)}_{x_1}$. Now the second one starts at $nx_1$ and stops once it exits the interval $(0, nx_2)$. Recursively, if the first $k$ excursions have not touched zero, the $(k+1)$st excursion starts at $nx_{k-1}$ and stops when it exits $(0, nx_k)$. If it does not exit through zero, define $x_{k+1}= x_k + \mathcal{I}^{(n)}_{x_k}$. This continues until we find an excursion that hits zero. By induction, every finite-dimensional distribution, starting from $1$, in this sequence converges in law thanks to Theorem \ref{stoppedage}. The fact that we eventually stop is guaranteed by the fact that the forest corresponding to the JCCP is critical. 

The rescaled measure on the street process of each individual excursion converges. A somewhat tedious argument but vey similar to the ones above show a joint convergence. 
\end{proof}

Note that state space of a street is $\mx_{\rr^+ \times \rr^+} \times \rr^+ \times \rr^+$ which is a complete separable metric space and hence allows a regular conditional probability. Now, for any two levels $a_0 < a_1$, the conditional distribution of $\street(a_1)$ given $\street(a_0)$ can be calculated any initial parameter $(v_0, v_1)$ by the Markov property that the limiting street derives from the finite trees. Hence one can define unambiguously the following transition probability.

\begin{figure}[t]
\centering
\includegraphics[width=5in, height=3.5in]{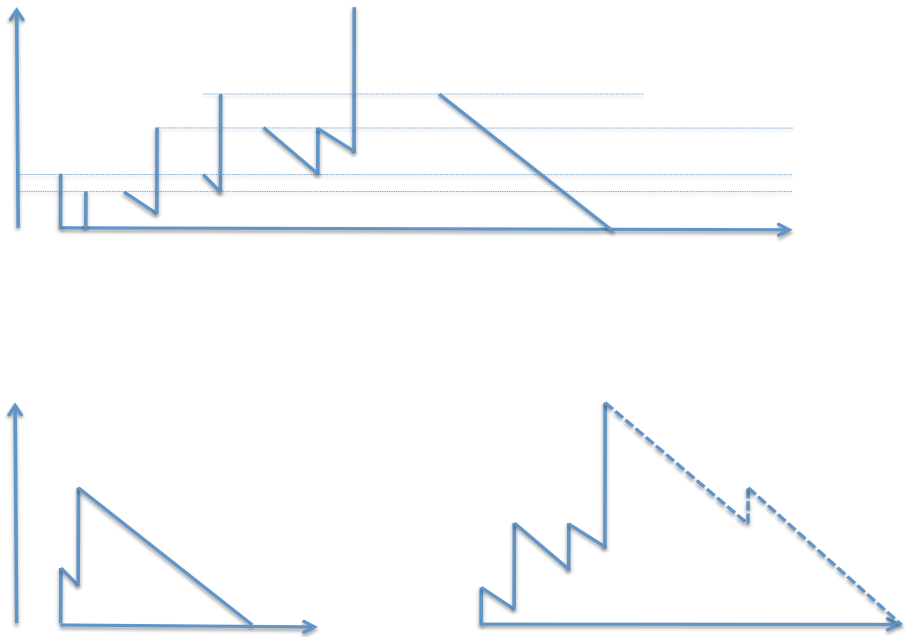}
\caption{Path decomposition of the reduced JCCP}
\label{redJCCP4}
\end{figure}
\bigskip
\bigskip

\begin{defn}
For any pair of levels $a_1, a_0 > 0$, we define $\regtran(a_0, a_1)$ to be the regular conditional distribution of $\street(a_1)$ given $\street(a_0)$. 
\end{defn}

\comment{
\begin{lemma}
The family of conditional distributions $\{  \regtran(s,t),\; 0 < s < t < \infty  \}$ is a Markov semigroup on the state space $\mx_{\rr^+ \times \rr^+} \times \rr^+ \times \rr^+$.
\end{lemma}

\begin{proof}
This follows from the fact that for any triplet $0 < a_0 < a_1 < a_2$, the joint distribution of $(\street_n(0), \street_n(a_0), \street_n(a_1), \street_n(a_2))$ satisfies the Markov property. Additionally their joint and the conditional distributions all converge to the corresponding limits. 
\end{proof}
}

To complete the description of the process of streets we will need a canonical \textit{entrance distribution}. This definition is standard for Markov processes. A family of $\sigma$-finite measures $\{ \mu(a) ,\; a> 0 \}$ will be called an entrance distribution for the semigroup $\regtran(\cdot, \cdot)$ if $\int \regtran(a_0, a_1) d \mu(a_0)= d\mu(a_1)$ for every pair of levels $a_0 < a_1$. 
However, instead of getting into technical details regarding Markov processes and semigroups we simply define a canonical distribution for the initial jumps $(V_0, V_1)$ for the finite reduced JCCP and take a limit.

Unfortunately this entrance distribution cannot be obtained directly by considering a sapling with exactly two leaves (the $Y$-tree) in Figure \ref{excursionbegins}. The reason being that the one leaf subtrees are not \textit{large enough} to be significant under our scaling. Heuristically, one of the subtrees containing these leaves will die too soon, and, although the other subtree might remain large, we are left with a trivial point process. Hence the entrance law has to be guessed from the  excursions of the reduced JCCP itself.

The solution to this problem follows from a path decomposition of the final image in Figure \ref{redJCCP1} which we have described in Figure \ref{redJCCP4}. The top image in Figure \ref{redJCCP4} is identical to the final image in Figure \ref{redJCCP1}. Notice, by independent increment and the strong Markov property of the compound Poisson process, the reduced JCCP can be seen as being constituted of two paths of the JCCP (not-reduced) interwoven together, until one of the paths hit zero.   

More specifically, we start with two paths of the JCCP that start with a jump at time zero and stopped when they return to zero. We call the path on the left of the second row of Figure \ref{redJCCP4} to be the $0$-JCCP (since it starts with a jump of size $V_0$), and we call the other one to be the $1$-JCCP (since its starts with jump size $V_1$). The way they are interwoven is clear: we start with the current supremum of the $0$-JCCP and observe the $1$-JCCP until it exceeds this level and note its end level. Then we start the $0$-JCCP until it exceeds the previous level and so on. At any moment if any of the two JCCP's hit zero we stop both the processes. In our given image, the $0$-JCCP hits zero first, and hence we draw the rest of the path for the $1$-JCCP in dots to show that it does not exist in the original reduced JCCP. 

There is only one way of defining a natural excursion of the $0$ and $1$ JCCP. Consider the compound Poisson process (with drift) $X$ as defined in \eqref{cpp} starting at zero. Let $I_t=\inf\{ X_s,\; 0\le s\le t  \}$ denote the running infimum. Consider the process reflected at this running infimum: $X_t - I_t$, for $t\ge 0$. Then, this naturally decomposes the paths of $X$ in a sequence of excursions whose paths are similar to paths of JCCP $0$ and $1$. Another way to describe this excursion is start with a jump of size $I_0+J_0$ as given in Lemma \ref{ladderheight} and then follow the process $X$ until we return to zero.  

Thus the canonical choice of $V_0, V_1$ is a pair of iid samples with a distribution function
\[
\begin{split}
P(V_0 > x) &= P(V_1 > x) = P_0( I_0 + J_0 > x)\\
&=x\lbar(x) + \gbar(x), \qquad x\ge 0.
\end{split}
\] 
Note that
\[
\begin{split}
\lim_{n\rightarrow \infty} nP\left( V_0 > nv_0, V_1 > nv_1\right)&=\left( 2^{-3/2} + 2^{-1/2} \right)^2 \frac{1}{\sqrt{v_0v_1}}= \frac{3}{2\sqrt{2v_0}} \cdot \frac{3}{2\sqrt{2v_1}}.
\end{split}
\]
which gives us the rate function for the initial jumps $(V_0, V_1)$.

\begin{lemma}\label{backward}
For any level $a >0$ and any $n \in \mathbb{N}$, consider an excursion of the reflected process $X-I$, starting from zero, conditioned to hit $na$ and stopped once it goes below zero. 
\begin{enumerate}
\item[(i)] Consider the joint distribution of the point process $\Pi_n$ of undershoots below level $na$, indexed in the reverse order that they appear (from right to left), without counting the final undershoot, $J_n(a)$ (the first undershoot in the forward order).  Then, as $n$ tends to infinity, the distribution of $(\Pi_n, J_n(a))$, suitably rescaled, converges to a probability measure $\pi(a)$ on $\mx_{\rr^+\rr^+}\times \rr^+$. In particular, suppose $(\Pi(a), J(a)) \sim \pi$, then $\Pi(a)$ and $J(a)$ are independent. $\Pi(a)$ is compactly supported with length Exponential with mean 
\[
\lambda(a)=\frac{3}{2\sqrt{2\pi}}\sqrt{a}.
\]
Moreover $\Pi(a)$ is a Poisson point on $(0, \lambda(a))\times \rr^+$ with a rate $\leb \times g^*_a(v)dv$, where $g^*$ was defined in \eqref{whatisgstar}.

\item[(ii)] Consider $Y_0$ and $Y_1$ to be two iid samples from the reflected process. Define the following events
\[
\begin{split}
E_0 (a) &:= \left\{   \sup Y_0 > na - J_n(a; 0)  ,\; \sup Y_1 > na \right\},\\
E_1 (a) &:= \left\{   \sup Y_1 > na - J_n(a; 1)  ,\; \sup Y_0 > na \right\}.
\end{split}
\]
where $J_n(a; i)$ is the first undershoot below level $na$ for the process $Y_i$.

Then, the following limit exists:
\eq\label{whatissigmaa}
\begin{split}
\sigma(a)&:=\lim_{n\rightarrow \infty} n P \left(  E_0(a) \cup E_1(a)   \right).
\end{split}
\en
\end{enumerate}
\end{lemma}

\begin{proof}
The proof of the lemma follows from very similar calculations done before. The length of the limiting point process follows since, starting at $na$, the process $X$ has $1/W(na)$ chance of hitting zero before $na$ again. Thus, after the first upcrossing of $na$ by $X$, there are Geometric$(1/W(na))$ many upcrossings of level $na$, each of density $r_a$ described in Lemma \ref{probcross}. The Geometric distribution converges to Exponential with the stated mean and the density converges to $g_a^*$. 
\end{proof}

\begin{thm} For any level $a$, let $\street(a)=(\nu,I,R)$ denote a (possibly empty) street. The following serves as an entrance distribution $\mu(a)$ for the semigroup $\regtran(\cdot, \cdot)$ described above. Under $\mu(a)$, the street is non-empty with rate $\sigma(a)$, where $\sigma(a)$ is defined in \eqref{whatissigmaa}. Let $\streetmarginal(a)$ be the probability measure on streets, conditioned to be non-empty.
\begin{enumerate}
\item[(i)] Under $\streetmarginal(a)$, the random variable $R$, has the distribution of $J(a)$ in Lemma \ref{backward}.
\item[(ii)] Given that $\nu$ is non-empty, it's distribution is $\pi(a)$ as given in Lemma \ref{backward}.
\end{enumerate}
\end{thm}

\begin{proof} 

Now, if the initial parameters $(V_0, V_1)$ are iid then clear the $0$-JCCP and the $1$-JCCP are iid too. 
Let $M_0(s)$ and $M_1(s)$ denote the running supremum of the $0$ and $1$ JCCP respectively, where $M_0$ and $M_1$ will denote the overall supremum. Consider the event, $E_a$, that the reduced JCCP crosses a level $a > 0$. Now, there are three possible cases by which this can happen.
\begin{enumerate}
\item[(i)] Both $\{ M_0 > a\}$ and $\{ M_1 > a \}$.
\item[(ii)] Either $\{ M_1(\infty) > a  \}$ and $ M_1(\tau_a-) < M_0 < a$,
\item[(ii')] or, same as (ii), with the role of $M_0$ and $M_1$ reversed.  
\end{enumerate}

The three events above are disjoint and their respective asymptotic densities (when we replace $a$ by $na$) can be easily calculated from Lemma \ref{overshoot}. 

Now, when $M_0 \wedge M_1 > a$, the point process of undershoots below a level $a$ (except the first one) is non-empty. Note that this set of undershoots below level $a$ for the reduced JCCP is the union of the the undershoots for the $0$ and $1$ JCCPs, suitably arranged.  Since they are all iid, this gives us the limiting distribution of $\pi$.  

In either case, at any level $a$, the distribution of the street is given by the appropriate distributions and we can take limits. As before, the convergence of finite dimensional distributions show that the claimed distribution is indeed an entrance measure. 
\end{proof}

\section{Dynamics of the mailman} Finally we arrive at the limiting dynamics of the mailman under the Poissonized Markov chain. At this point we do not prove any finite to continuum convergence. We simply construct a limiting process with the ingredients derived in the last Section, the laws $\{ \streetmarginal(a), \; a\ge 0\}$, and the derived transition kernel $\regtran(\cdot, \cdot)$.
\bigskip

\noindent\textbf{Construction of a random mailman.} We are going to sequentially construct an exchangeable sequence of mailmen $( X(1), X(2), X(3)\ldots)$. 

Fix an initial level $a> 0$. Let $\street_1(1)=(\nu_1, I_1, R_1)$ be a sample from the distribution $\Gamma(a)$. Let $(X_1(1), \varsigma_1(1))$ be distributed according to this street. Now, given $X_1(1)$, let $Y_1(1)$ denote the age of the atom at $X_1(1)$. This is merely the mass of $\nu_1$ at $X_1(1)$ when $\varsigma_1(1)=1$, and $R_1$ when $\varsigma_1(1)=0$. Let $\mathcal{F}_1(1)$ be the $\sigma$-algebra generated by $(\street_1(1), X_1(1), \varsigma_1(1))$.

Now, given $\mathcal{F}_1(1)$, generate $\street_2(1)$ according to $\Gamma(Y_1(1))$, and define $X_2(1)$, $\varsigma_2(1)$, and $Y_2(1)$ analogously. Let $\mathcal{F}_2(1)$ be the $\sigma$-algebra generated by all these new random elements and $\mathcal{F}_1(1)$. By induction, given $\mathcal{F}_k(1)$, generate $\street_{k+1}(1)$ according to $\Gamma(Y_k(1))$, and update $\mathcal{F}_{k+1}(1)$. Let $\mathcal{F}_\infty(1)=\sigma\left(  \cup_k \mathcal{F}_k(1) \right)$.

Let the joint law of the family $((X_k(1), \varsigma_k(1), \street_k(1)), \; k=1,2,\ldots )$, starting at an initial age $a$, be denoted by $\hugelaw(a)$.

Now, to generate $X(2)$ given $\mathcal{F}_\infty(1)$, start by generating $(X_1(2), \varsigma_1(1))$ as an independent sample from $\street_2(1):=\street_1(1)$. Let $\mathcal{F}_1(1)$ be the enlarged filtration with the additional information. Now we proceed by induction. Define the stopping time
\[
m(1,2)= \min \left\{ n\ge 0: \;  (X_n(1), \varsigma_n(1)) \neq (X_n(2), \varsigma_n(2)) \right\}.
\] 
Then, then if $\{ k < m(1,2) \}$, given the history, generate $(X_{k+1}(2), \varsigma_{k+1}(2))$ from $\street_{k+1}(2):=\street_{k+1}(1)$. For the rest of the sequence, given $\mathcal{F}_{m(1,2)}(2)$, the shifted family 
\[
((X_{m(1,2)+k(1)}, \varsigma_{m(1,2)+k(1)}, \street_{m(1,2)+k(1)}), \; k=1,2,\ldots ),
\]
is distributed as $\hugelaw(Y_{m(1,2)}(2))$.

The construction of the entire exchangeable family $(X(1), X(2), \ldots)(a)$ now follows by an obvious induction. The law of this family on the space of random sequences is the address protocol at level $a$.  
\bigskip

\noindent\textbf{One-step transition of the random mailman.} Now suppose we are given a pair of levels $a_0 < a_1$ and an instance of the family 
\[
\left\{ (X(1),\varsigma(1), \street(1)),\;  (X(2), \varsigma(2), \street(2)),\;  \ldots \right\}(a_0).
\] 
Let $\mathcal{G}(a_0)$ be the $\sigma$-algebra generated by the entire above family. We now describe how to generate an iid sequence of mailmen at level $a_1$, given $\mathcal{G}(a_0)$.
\bigskip

To keep our notation simple, we will denote by $X', \varsigma', \street'$ to denote the corresponding quantities at level $a_1$. We will enlarge the $\sigma$-algebras naturally as we go. Recall the transition kernel $\regtran(a_0, a_1)$ from the last section. The law of $\street'_1(1)$ given the entire history is equal to the law of the $\street'_1(1)$ given $\street_1(1)$ which is given by $\regtran(a_0, a_1)$. Call $\street_1(1)$ to be the parent of $\street'_1(1)$.

Let $(X'_1(1), \varsigma'_1(1))$ be distributed according to this street, and let $Y'_1(1)$ denote of $X'_1(1)$. 
Now, there are two cases to consider:
\begin{enumerate}
\item[Case (i)] if $Y'_1(1) < a_1 - a_0$, then the subtree corresponding to this atom did not exist at level $a_0$. Hence, we generate $\street'_2(1)$ as a sample from $\Gamma(Y'_1(1))$. Generate the mailman as an independent sample from this street.
\medskip

\item[Case (ii)] if $Y'_1(1) > a_1 - a_0$, then, by the restriction consistency property of the family $\{ \Gamma(a), \; a\ge 0\}$, there is a unique atom on $\street'_1(1)$ such that the age of that atom is exactly $Y'_1(1)-a_1+a_0 > 0$. Now since this atom has positive probability, there exists a $n_0$ such that $Y_{1}(n_0)$ is equal to this age. The distribution of $\street'_2(1)$, given the entire history, is taken to be the distribution of $\street'_2(1)$ given $\street_2(n_0)$, which is a  sample distributed according to $\regtran(Y_{1}(n_0), a_1)$. Generate the mailman as an independent sample from this street.

We call $\street_2(n_0)$ to be the parent of $\street'_2(1)$. Note that this definition does not change if $n_1$ is another index such that $Y_{1}(n_0)=Y_{1}(n_1)$. 
\end{enumerate}
\bigskip

For the induction step, suppose $\street'_{k}(1)$ has been generated and $Y'_k(1)$ is the age of the mailman $(X'_k(1), \varsigma'_k(1))$. Then, we consider cases (i) and (ii) as above. For case (i), $\street'_{k+1}(1)$ is generated according to $\Gamma(Y'_k(1))$. 

For case (ii), suppose that the difference $Y'_k(1) - a_1 + a_0 > 0$. Then, necessarily, the ages of each of $X'_1(1), \ldots, X'_k(1)$ is greater than $a_1 - a_0$. Since, if $X'_i(1)$ has age less than $a_1 - a_0$, then so does $X'_{i+k}(1)$ for all $k\in \mathbb{N}$.

Then, every street $\street'_1(1), \street'_2(1), \ldots, \street'_k(1)$ has a corresponding parent at level $a_0$. And, in fact, almost surely there exists $n_k$ such that each of $\street_1(n_k), \street_2(n_k), \ldots, \street_k(n_k)$ is a parent of the corresponding streets at level $a_1$ and $Y_{k}(n_k)=Y'_k(1)-a_1+a_)$. The distribution of $\street'_{k+1}(1)$, given the history, is identical to the law of $\street'_{k+1}(1)$ given $\street_{k+1}(n_k)$ and is given by $\regtran(Y_{k}(n_k), a_1)$. The mailman is an independent sample from this street. By induction this generates the process $X'(1), \varsigma'(1), \street'(1)$ given the history so far. 

Now the reader can guess how to generate $X'(2), \varsigma'(2), \street'(2)$. If we hit an age of an atom that has already appeared in $X'(1)$, the next street is kept unchanged, and the mailman is an independent sample from that street. If we hit an atom that has not appeared in $X'(1)$ we proceed by case (i) or (ii) as described above. And so on for the rest of $(X', \varsigma', \street')$.
\bigskip

This defines a joint distribution of countable collection of random sequences at levels $a_0$ and $a_1$. Suppose these random sequences are almost surely in $\mathbb{L}^1$, they are the set representations of a pair of continuum trees which has evolved according to the dynamics of the Poissonized Aldous Markov process. 

Unfortunately, although it seems very intuitive that these are in $\mathbb{L}^1$, we cannot come up immediately with a quick argument. Hence we leave this as a conjecture to be settled in a follow up paper. 

\section*{Concluding remark} The invariant distribution for the finite Aldous Markov chain is the uniform distribution on binary trees. Hence it is intuitive that the Brownian CRT is the invariant distribution for the limiting diffusion on tree. However, perhaps the most unsatisfying aspect of this article that the author cannot see where the CRT is hidden. There are several tell-tale hints, especially in the appearance of the Stable$(1/2)$ subordinator. However, a clear connection with the CRT is missing at this point.

\section*{Acknowledgment} I am grateful to David Aldous for suggesting me this problem on a nice Spring visit to Berkeley. David, Jim Pitman, and Chris Burdzy have been most generous with their time while discussing with me during the progress of this project. My thanks to all of them.

\bibliographystyle{alpha}

\end{document}